\documentclass[12pt, reqno, a4paper]{amsart}
\usepackage[margin=1in]{geometry}
\usepackage{amsmath,amssymb}
\usepackage{enumitem}
\usepackage{graphicx}
\usepackage{caption}
\usepackage{mathrsfs}
\usepackage{mathtools}          
\usepackage{esint}
\usepackage{enumitem}   
\usepackage{mathtools}
\usepackage{amssymb}
\usepackage{soul}
\usepackage{xcolor}
\usepackage{cancel}
\usepackage{amsthm}

\makeatletter
\renewcommand\subsubsection{\@startsection{subsubsection}{3}%
  \z@{1ex \@plus .2ex \@minus .1ex}{-1em}%
  {\normalfont\normalsize\itshape}}
\makeatother

\usepackage{hyperref}

\DeclareMathOperator*{\argmin}{arg\,min}

\DeclareGraphicsExtensions{.pdf,.png,.jpg}

\def\Xint#1{\mathchoice
{\XXint\displaystyle\textstyle{#1}}%
{\XXint\textstyle\scriptstyle{#1}}%
{\XXint\scriptstyle\scriptscriptstyle{#1}}%
{\XXint\scriptscriptstyle\scriptscriptstyle{#1}}%
\!\int}
\def\XXint#1#2#3{{\setbox0=\hbox{$#1{#2#3}{\int}$ }
\vcenter{\hbox{$#2#3$ }}\kern-.6\wd0}}

\def\dashint{\Xint-}

\newtheorem{theorem}{Theorem}[section]
\newtheorem{proposition}[theorem]{Proposition}  

\newtheorem{lemma}[theorem]{Lemma}

\theoremstyle{definition}

\newtheorem{example}[theorem]{Example}

\usepackage{color}

\theoremstyle{remark}
\newtheorem{remark}[theorem]{Remark}
\numberwithin{equation}{section}

\newcommand{\R}{{ \mathbb{R}  }}

\newcommand{\N}{{ \mathbb{N}  }}
\newcommand{\p}{{ \partial  }}
\newcommand{\al}{{ \alpha }}

\newcommand{\Om}{{ \Omega  }}
\newcommand{\lam}{{ \lambda  }}

\newcommand{\calM}{{\mathcal M}}\newcommand{\calN}{{\mathcal N}}
\newcommand{\calE}{{\mathcal E}}

\newcommand{\calH}{{\mathcal H}}

\newcommand{\del}{\delta}
\newcommand{\gam}{\gamma}

\newcommand{\vphi}{\varphi}

\newcommand{\na}{\nabla}

\newcommand{\T}{\tau}

\newcommand{\Th}{\Theta}
\newcommand{\z}{\zeta}	
\newcommand{\ra}{\rightarrow}	
\newcommand{\da}{\downarrow}

\newcommand{\OM}{\overline{\Omega}}
\newcommand{\ep}{\varepsilon}

\newcommand{\wra}{\rightharpoonup}

\newcommand{\lt}{\left}
\newcommand{\rt}{\right}

\newcommand{\lag}{\langle}	
\newcommand{\rag}{\rangle}

\def\dashint{\Xint-}

\def\BigRoman{\uppercase\expandafter{\romannumeral\number\count
255 }}
\def\Romannumeral{\afterassignment\BigRoman\count255=}

\author{Dongkwang Kim}
\address{Dongkwang Kim: Department of Mathematical Sciences \newline Ulsan National Institute of Science and Technology
(UNIST), Ulsan, Republic of Korea}
\email{dkkim@unist.ac.kr}

\author{Dowan Koo}
\address{Dowan Koo: Mathematical Institute, University of Oxford, \newline  Oxford OX2 6GG, United Kingdom}
\email{dowan.koo@maths.ox.ac.uk}

\author{Geuntaek Seo}
\address{Geuntaek Seo:  Department of Mathematics\newline POSTECH, Pohang 37673,  Republic of Korea}
\email{gtseo@postech.ac.kr}

\begin{document}
\captionsetup[figure]{labelfont={rm},labelformat={default},labelsep=period,name={Fig.}}

\title[A gradient flow for the PME with Dirichlet boundary conditions]{A gradient flow for the porous medium equations with Dirichlet boundary conditions}

\begin{abstract}
We consider the gradient flow structure of the porous medium equations with non-negative constant Dirichlet boundary conditions. We construct weak solutions to the equations via the minimizing movement scheme by considering an entropy functional with respect to $Wb_2$ distance, which is a modified Wasserstein distance introduced by Figalli and Gigli [J. Math. Pures Appl. 94, (2010), pp. 107-130]. 
Furthermore, the constructed solutions are characterized as curves of maximal slope in a suitable sense.
\end{abstract}

\maketitle

\section{Introduction}
The porous medium equation (PME) has been extensively studied in mathematics and physics over the past decades. Many works have investigated its properties, particularly from the perspective of partial differential equation analysis (see \cite{Vaz1, Vaz2} and references therein). In this paper, we aim to establish the gradient flow formulation of the PME with Dirichlet boundary conditions within the framework of the optimal transport.

\subsection{Historical background}
We briefly review the historical background of Wasserstein gradient flows in the context of evolutionary PDEs, as well as recent work addressing Dirichlet boundary problems.

\subsubsection*{Development of Wasserstein gradient flow}  
A foundational result in the study of Wasserstein gradient flows is the work of Jordan, Kinderlehrer, and Otto \cite{JKO}. The authors demonstrated that solutions to the linear Fokker-Planck equation
$\partial_t \rho = \Delta \rho + \nabla \cdot (\rho \nabla V)$
can be obtained as the limit of time-discrete variational approximations (the so-called JKO scheme) for the energy functional 
$\int \rho \log\rho + V\rho dx$
with respect to the $2$-Wasserstein distance $W_2$, a complete metric on the set of probability measures with finite second moment.
Furthermore, Otto \cite{O} interpreted PME $\p_t \rho = \Delta \rho^\al$, with $\al>1$, as a Wasserstein gradient flow for the internal energy $\int \rho^\alpha/(\alpha-1) dx$, by generalizing concepts from Riemannian geometry to the Wasserstein space. 
After these two seminal works, the theory of Wasserstein gradient flow for various evolutionary PDEs has been successfully developed in bounded domains with Neumann boundary conditions and in the whole space setting; see e.g. \cite{AGS, CMV1, FM, San0}.
For related results in the Riemannian manifold setting, see e.g. \cite{CMV2, DMO, E, OS, OW, Stu, Vil2}.

\subsubsection*{The work of Figalli and Gigli}
The framework of Wasserstein gradient flows in the space of probability measures presupposes mass conservation. This approach is thus not applicable to evolutionary PDEs that do not conserve mass, such as Dirichlet boundary problems.
Figalli and Gigli \cite{FG} introduced a new transportation distance between two non-negative measures on a bounded open set $\Omega \subset \R^d$ with possibly different masses. Their idea was to interpret the boundary $\p \Omega$ as an infinite `reservoir' of mass so that the mass can be exchanged through the boundary.

Precisely, the Figalli-Gigli distance $Wb_2$ is defined on a subset of non-negative (possibly infinite) Borel measures on $\Omega$: 
\begin{equation}\nonumber
\mathcal{M}_2(\Om):= \Big\{\mu \in \mathcal{M}_{+}(\Om): \int_{\Om} d^2(x,\p \Om) d\mu(x)<\infty   \Big\}.
\end{equation}
For $\mu, \nu \in \mathcal{M}_2(\Om)$, the distance $Wb_2(\mu,\nu)$ is defined as
\begin{equation}\label{metric}
Wb_2^2(\mu,\nu):= \inf_{\gam \in ADM(\mu,\nu)} \int_{ \OM \times \OM} |x-y|^2 d\gam(x,y),
\end{equation}
where the set of admissible couplings $ADM(\mu, \nu)$ is a collection of positive measures $\gam$ on $\overline{\Om} \times \overline{\Om}$ satisfying:
$$
(\pi^1_{\#} \gam) |_{\Om}=\mu, \quad (\pi^2_{\#} \gam) |_{\Om}=\nu,
$$
where $\pi^i$, $i=1,2$, denote the projections of $\OM \times \OM$ onto its components. Various properties regarding $Wb_2$ metric are recorded in Section \ref{ssec:wb2}.

As an application of $Wb_2$ distance, they showed that the curve obtained by the minimizing movement scheme for the Boltzmann entropy functional
\begin{equation}\label{eq:Bol}
\calH(\rho) =  \int_{\Om} [\rho \log \rho - \rho +1 ] dx
\end{equation}
converges to the solution of the heat equation with the Dirichlet boundary condition 
$$
\rho|_{\p \Om} = 1.
$$

\subsection{Main results}
We study the following PME with Dirichlet boundary conditions,
interpreting \eqref{PME} as a gradient flow formulated via the minimizing movement scheme with the $Wb_2$ distance:
\begin{equation}\label{PME}
  \begin{cases}
                                   \p_t \rho= \Delta \rho ^{\al} & \text{in $\Om \times (0,\infty)$,} \\
                                   \rho(x,0)=\rho_0(x)  & \text{in $\Om$,} \\
  \rho(x,t)=\lam & \text{on $\p \Om \times (0,\infty)$.}
  \end{cases}
\end{equation}
Here, $\al>1$ and $\lambda \ge 0$ are given constants, and $\Om$ is a bounded domain in $\R^d$ with Lipschitz boundary.

\subsubsection*{Overview and ideas}
To set up the minimizing movement scheme, the first step is to find a suitable entropy functional such that the limit obtained through the scheme can solve the equation $\eqref{PME}$. To this end, we slightly adjust the entropy $\int \rho^\alpha/(\alpha-1) dx$.
For a given $\mu \in \mathcal{M}_2(\Om)$, we define the associated internal energy functional $\calE(\mu) = \calE_{\al,\lambda}(\mu)$ as follows:
\begin{equation*}
  {\calE}(\mu):=
  \begin{cases}                         \displaystyle{\int_{\Om}}  
 U(\rho) dx& \text{if $\mu=\rho  \mathcal{L}^d |_{\Om}$}, \\
                                   \infty & \text{else,} 
  \end{cases}
\end{equation*}
where
\[
U(\rho):=  U_{\al,\lam}(\rho)= \frac{\rho^{\al} - \al \lam^{\al-1}\rho}{\al-1}   + \lam^\al. 
\]
We will omit the subscripts $\al, \lam$ for brevity. Note that $\calE(\mu) \geq 0$ and the integrand $U(\rho)$ attains its minimum at $\rho=\lam$.
We remark that, for the heat equation, the integrand in \eqref{eq:Bol} can be tilted to $\rho \log \rho -(\log \lam +1)\rho$ to impose Dirichlet boundary condition $\rho|_{\p \Om} = \lam$; however, this approach is restricted to strictly positive case $\lam>0$.

Let us introduce the minimizing movement scheme: for a given $\mu_0 \in \mathcal{M}_2(\Om)$ with ${\calE}(\mu_0)<\infty$, and for each time step $\T>0$, we define recursively  
\begin{equation}\label{argmin}
\mu_k^\T \in \argmin_{\mu \in \mathcal{M}_2(\Om)} \Big\{ {\calE}(\mu ) + \frac{Wb_2^2(\mu_{k-1}^\T, \mu)}{2\T} \Big\}, \quad \mu_0^\T:=\mu_0, \quad \textrm{for} ~ k=1,2,....
\end{equation}
The variational problem \eqref{argmin} has a unique solution (see Lemma \ref{minimizer}), and thus the following \textit{discrete solution} is well-defined:

\noindent
\begin{equation}\label{DS}
\mu^\T(t):=\mu_n^\T \quad \textrm{for} \quad t \in ( (n-1) \T, n\T].
\end{equation}

We highlight one of the major differences between the heat equation and the PME. Unlike the heat equation, which has an infinite propagation speed, the solutions to the PME propagate at a finite speed. 
 In particular, the solution to the PME \eqref{PME} may vanish near the boundary for a short time if initial data vanish on a subset of $\Om$. 
 This behavior poses a significant challenge in analyzing the trace of the discrete solution $\mu^\tau$. Indeed, perturbative arguments used to determine the trace of $\mu^\T$ typically require the strict positivity of the discrete solution, especially when $\lambda > 0$. However, this condition may fail in the PME due to its degeneracy.
In Lemma \ref{bdry_diffusion}, we prove the positivity of the discrete solution $\mu^\tau$ near the boundary by carefully exploiting the features of the $Wb_2$ metric, even without requiring the positivity of $\mu_0$ near $\p \Om$.  This enables us to obtain quantitative estimates for $\mu^\T$ near the boundary, as shown in Proposition \ref{DS_regularity}.

\bigskip

\subsubsection*{Statement of the main result}
Before stating the main theorem, we briefly recall the metric
notions needed to formulate its final assertion; see \cite{AGS}
for the general framework.
We say
\(
\mu:[0,\infty)\to\mathcal M_2(\Omega)
\)
belongs to
\(
AC^2_{\mathrm{loc}}
\bigl([0,\infty);\mathcal M_2(\Omega)\bigr)
\)
if, for every \(T>0\), there exists \(m \in L^2(0,T)\) such that
\[
Wb_2(\mu(s),\mu(t))
\leq
\int_s^t m(r)\,dr
\qquad
\text{for all }0\leq s\leq t\leq T.
\]
For such a curve, the \textit{metric derivative} exists for a.e. \(t>0\)
and is given by
\begin{equation}\label{def_metric_derivative}
|\mu'|(t)
:=
\lim_{s\to t}
\frac{Wb_2(\mu(s),\mu(t))}{|s-t|}.
\end{equation}
For $\mu\in\mathcal{M}_2(\Omega)$, \textit{the slope of $\mathcal{E}$} at $\mu$ is defined by
\begin{equation}\label{def_of_slope}
|\na \calE|(\mu):=\limsup_{\nu \ra \mu} \frac{\Big({\calE}(\mu) -\calE(\nu) \Big)^+}{Wb_2(\mu,\nu)},
\end{equation}
where $\nu\to\mu$ with respect to $Wb_2$. Its relaxed slope is defined by
\begin{equation}\label{def_of_relaxed_slope}
|\nabla^-\mathcal{E}|(\mu)
:=
\inf\left\{
\liminf_{n\to\infty}|\nabla\mathcal{E}|(\mu_n):
\begin{array}{l}
\mu_n\rightharpoonup\mu\ \text{vaguely},\\
\sup_n\{Wb_2(\mu_n,\mu),\mathcal{E}(\mu_n)\}<\infty
\end{array}
\right\},
\end{equation}
where the vague convergence $\mu_n \rightharpoonup \mu$ means that $\int \varphi \, d\mu_n \to \int \varphi \, d\mu$ for all $\varphi \in C_c(\Omega)$.

Finally, a locally absolutely continuous curve $[0,\infty) \ni t  \mapsto \mu(t) \in \mathcal M_2 (\Omega)$ 
is called a \emph{curve of maximal slope} for $\mathcal E$ with respect
to the relaxed slope $|\nabla^-\mathcal E|$ if there exists a
non-increasing map
$\varphi:[0,\infty)\to\mathbb R$ such that
\[
\varphi(t)=\mathcal E(\mu(t))
\qquad\text{for a.e. }t\geq0,
\]
and
\begin{equation}\label{eq:CMS}
\varphi'(t)
\leq
-\frac12|\mu'|^2(t)
-\frac12|\nabla^-\mathcal E|^2(\mu(t))
\qquad\text{for a.e. }t\geq0.
\end{equation}

We now present our main result. 
Our main result concerns the construction of a weak solution of \eqref{PME} using the minimizing movement scheme with the Figalli-Gigli distance. 
Moreover, the resulting curve is shown to be a curve of maximal slope in a suitable sense.

\begin{theorem}\label{main}
Suppose that $\mu_0:=\rho_0 \mathcal{L}^d |_\Om \in \mathcal{M}_2(\Om)$ and ${\calE}(\mu_0)<\infty$. Then we have
\begin{enumerate}
\item[\emph{(i)}] 
\emph{Convergence:} For every vanishing sequence $(\T_k)_{k \in \N}$, there exists a subsequence (not relabeled) such that
for every $t \geq 0$, $\mu^{\T_k}(t)$ defined in \eqref{DS} converges in $Wb_2$ to a measure $\mu(t)=\rho(t) \mathcal{L}^d |_\Omega \in \mathcal{M}_2(\Om)$ as $k \ra \infty$. In particular,
$\mu$ is $\frac{1}{2}$-H\"{o}lder continuous with respect to $Wb_2$. Moreover, the sequence also converges in $L_{loc}^\al \big( (0,\infty); L^\al(\Om) \big)$.

\item[\emph{(ii)}]
\emph{PME equation:}
The limit curve $\mu(t)=\rho(t)\mathcal{L}^d |_\Omega$ satisfies the following weak formulation: for any $\z \in C_c^{\infty}(\Om)$ and $0 \leq t_1 \leq t_2 < \infty$,
\begin{equation}\label{weak_solution}
 \int_{\Om} \rho (t_2) \z dx  - \int_{\Om} \rho(t_1) \z dx
 =   \int_{t_1}^{t_2} \int_{\Om} \big(\rho(t)\big)^{\al}    \Delta \z dx dt.
\end{equation} 

\item[\emph{(iii)}]
\emph{Boundary condition:}
The function
\begin{equation}\label{regularity}
t \mapsto \rho^{\al- \frac{1}{2}}(t) - \lam^{\al - \frac{1}{2}}   \quad
\textrm{belongs to}  \quad
L^2_{loc}([0,\infty); H_0^1(\Om)).
\end{equation}

\item[\emph{(iv)}]
\emph{Curve of maximal slope:}
The curve $\mu$ belongs to $
AC^2_{\mathrm{loc}}
\bigl([0,\infty);\mathcal M_2(\Omega)\bigr)$.
Moreover, it is a curve of maximal slope for $\mathcal E$
with respect to its relaxed slope $|\nabla^-\mathcal E |$.
\end{enumerate}

\end{theorem}

\begin{remark}\label{CMS_news}
We make the following remark concerning part~(iv) of
Theorem~\ref{main}. In recent work, Erbar and Meglioli
\cite{EM24} established a Benamou-Brenier-type dynamic
characterization of the Figalli-Gigli distance.
Under their structural assumptions, they proved that the relaxed
slope is a strong upper gradient and characterized weak solutions
precisely as curves of maximal slope, despite the lack of geodesic
convexity of the energy (see also \cite[Remark~3.4]{FG}).
For the power-law energy considered here, their assumptions are satisfied
for $1<\alpha\leq 3/2$.
\end{remark}

We leave comments on related results in the literature. 
While the heat equation with homogeneous boundary was not considered in \cite{FG}, Profeta and Sturm \cite{PS} interpreted the heat flow with homogeneous Dirichlet boundary conditions as a gradient flow, introducing the concept of a charged probability measure and defining a Wasserstein-type metric between charged probability measures. 

Finally, we refer to the recent work of Quattrocchi \cite{Q24}. 
The author in \cite{Q24} develops the variational structure for the linear Fokker-Planck equation with general Dirichlet boundary conditions, extending the results of \cite{FG} and \cite{M18}. The author shows that the solutions can be obtained via a modified JKO scheme for a relative entropy, using measure data over $\OM$ and replacing $Wb_2$ with a transport cost functional which is not a true metric.
In the one-dimensional setting, it is further proved that the solution corresponds to the curve of maximal slope within an appropriate space.
The formulations considered in this paper and in \cite{FG} are restricted to constant Dirichlet boundary data. It is left for future work to investigate whether an analogous variational formulation can be developed for nonlinear diffusion equations with general Dirichlet data.

\medskip
\subsubsection*{Plan of the paper}
In Section \ref{sec:pre}, we review the metric setting $(\calM_2(\Om),Wb_2)$ and related concepts, along with various analytical tools such as compactness and trace operators that will be used later.
Section \ref{sec:lemmas} collects several auxiliary lemmas that will be used in the proofs of the main results.
Section \ref{sec:bdry} focuses on the boundary analysis of discrete solutions.
Section \ref{sec:thm1} is devoted to the proof of Theorem \ref{main}.
In Appendices, we briefly discuss how our arguments can be extended to the drift-diffusion case. Moreover, we address the pointwise lower bounds of the metric derivative.

\vspace{3mm}
\section{Preliminaries}\label{sec:pre}

This section records some notions and properties that will be used throughout this paper.

\subsection{The space $(\calM_2(\Omega), Wb_2)$} \label{ssec:wb2}
Let us recall some metric properties regarding $(\mathcal{M}_2(\Om) , Wb_2)$:

\begin{proposition} \cite[Theorem 2.2, Proposition 2.7 and 2.9]{FG}

\begin{enumerate}
\item[\emph{(i)}]
The metric space $(\mathcal{M}_2(\Om), Wb_2)$ is a Polish space, and it is always a geodesic space (while $(\mathcal{P}(\Om),W_2)$ is geodesic if and only if $\Om$ is convex). 

\item[\emph{(ii)}]
The subset $\mathcal{M}_{\leq M}(\Om)$ of $\mathcal{M}_2(\Om)$ consisting of measures with mass less or equal to $M \in \R$ is compact. 

\item[\emph{(iii)}]
The convergence in $\mathcal{M}_2(\Om)$ can be understood in the following sense:
$Wb_2(\mu_n, \mu) \ra 0$ if and only if 
\begin{equation}\label{Wb2_iff}
\int_\Om \phi ~ d\mu_n \ra \int_\Om \phi ~d\mu,~ \forall \phi \in C_c(\Om) 
\quad \textrm{and} \quad \int_\Om d^2(x,\p \Om) d\mu_n(x) \ra \int_\Om d^2(x,\p \Om) d\mu(x).
\end{equation}

\item[\emph{(iv)}]
$Wb_2$ is lower semicontinuous with respect to weak convergence in duality with functions in $C_c(\Om)$.

\end{enumerate}
\end{proposition}

\begin{remark}\label{no_need_bdry^2}
Let $\mu, \nu \in \mathcal{M}_2(\Om)$ and let $\gam \in ADM(\mu,\nu)$. Then
$$
\gam - \gam|_{\p \Om \times \p \Om} \in ADM(\mu,\nu) \quad \textrm{and} \quad \mathcal{C}(\gam - \gam|_{\p \Om \times \p \Om}) \leq \mathcal{C}(\gam).
$$
In the problem of finding optimal plans, therefore we can assume that $\gam|_{\p \Om \times \p \Om} = 0$ without any restriction on $\gam \in OPT(\mu, \nu)$. 
Also, it is easy to check that if  $A \subset \Om$ is a Borel set with dist$(A, \p \Om) > 0$, then $\mu(A)$ is finite. Hence we have $\gam \big(( A \times \OM) \cup (\OM \times A)  \big)\leq \gam(A \times \OM) + \gam(\OM \times A) = \mu(A) + \nu(A)<\infty$.
These observations imply that the infimum in the right-hand side of \eqref{metric} can be attained by using the weak compactness of $\mathcal{A}:=\{ \gam : \gam \in  ADM(\mu, \nu)  \textrm{ satisfying } \gam|_{\p \Om \times \p \Om} = 0 \}$ in duality with functions in $C_c(\OM \times \OM \backslash \p \Om \times \p \Om)$ and the weak lower semicontinuity of $\gam \mapsto \mathcal{C}(\gam)$ (in fact, if $\mu$ and $\nu$ have finite mass, the weak compactness argument is rather simple: consider a minimizing sequence $(\gam_n)$ for \eqref{metric}.  Then the weak compactness of $(\gam_n)$  follows  directly from the fact that $\gam_n(\OM \times \OM)\leq \gam_n(\Om \times \OM) + \gam_n(\OM \times \Om) = \mu(\Om) + \nu(\Om)<\infty$. In this case, we have  a subsequence $(\gam_{n_k})$ weakly converging to $\gam^*$, i.e.,  $\int \phi d \gam_{n_k} \ra \int \phi d \gam^*$, $\forall \phi \in C(\OM \times \OM)$).
\end{remark}

\begin{remark}
It is worth mentioning the difference between  $(\mathcal{M}_2(\Om), Wb_2)$ and $(\mathcal{P}_2(\Om), W_2)$. It is easy to see that for two positive measures $\mu$ and $\nu$ with $\mu(\Om) = \nu(\Om)$, $Wb_2 (\mu, \nu) \leq W_2(\mu, \nu)$.
Note also that when we treat the metric $Wb_2$, $\p \Om$ plays the role of an infinite reserve whereas $\p \Om$ does not directly affect the distance $W_2$ between probability measures on $\Om$. More precisely, when finding the optimal plan between two measures with different masses, the part lacking mass is supplied from the boundary in terms of mass transport, and the excess part is sent to the boundary.
\end{remark}

\begin{example}
We consider 
\[
\mu:=\frac{1}{2}\del_{\ep} + \frac{1}{2}\del_{1-\ep},\quad \nu:=\frac{1}{2}\del_{\ep} + \frac{1}{4}\del_{1-\ep}
\]
with small $\ep>0$. These two measures with different masses $\mu, \nu$ belong to $\mathcal{M}_2\big( (0,1) \big)$ and one can find an optimal plan $\gam$ as follows: from the point of view of minimizing the cost, the mass $1/2$ of $\mu$ located at $\ep$ does not need to be moved to any point except $\ep$. Similarly, 1/4 mass of $\mu$ located at $1-\ep$ does not need to be moved at any point except $1-\ep$, and finally, the remaining 1/4 of the mass at position $1-\ep$ is sent to the nearest boundary $1$. This describes an optimal plan $\gam$. 
Even if the masses are the same, this phenomenon continues: we define
\[
\mu^*:=\del_{\ep},\quad \nu^* := \del_{1-\ep}
\]
with a small $\ep>0$. Obviously,  $\mu^*, \nu^* \in \mathcal{M}_2\big( (0,1) \big)$. Similar to the previous example, an optimal plan is $\gam^*=\del_\ep \otimes \del_0  + \del_1 \otimes \del_{1-\ep}$, which is significantly different from the optimal plan in $W_2$, namely $\del_\ep \otimes \del_{1-\ep}$. See \cite{FG} for more details of such a metric. 
\end{example}

\subsection{Analytical tools}
We recall Theorem 4.9 in \cite{FM}, which is equivalent to the generalized Aubin-Lions lemma proved in \cite[Theorem 2]{RS}.
\begin{proposition}\label{ALRS}
On a Banach space $X$, let be given
\begin{enumerate}[label=\textbullet]
\item a normal coercive integrand $\mathscr{F}: X \ra [0,\infty]$, i.e.,  $\mathscr{F}$ is lower semicontinuous, and its sublevels are relatively compact in $X$;

\item a pseudo-distance $g: X \times X \rightarrow [0,\infty]$, i.e., $g$ is lower semicontinuous, and $g(\rho, \eta)=0$ for any $\rho , \eta \in X$ with $\mathscr{F}(\rho) <\infty$, $\mathscr{F}(\eta) <\infty$ implies $\rho=\eta$.
\end{enumerate}
Let further $U$ be a set of measurable functions $u : (0,T) \ra X$ with a fixed $T>0$. Under the hypotheses that 
\begin{equation}\label{time_H^1}
\sup_{u\in U} \int_0^T \mathscr{F}(u(t)) dt <\infty
\end{equation}
and
\begin{equation}\label{AL_limsup}
\lim_{h \da 0} \sup_{u \in U} \int_0^{T-h} g(u(t+h) , u(t))dt =0,
\end{equation}
U contains an infinite sequence $(u_n)_{n \in \N}$ that converges in measure (with respect to $t \in (0,T)$) to a limit $u^*:(0,T) \ra X$. 
\end{proposition}

\begin{remark}
At the conclusion of the above proposition, the convergence in measure means that
$$
\lim_{n \ra \infty} 
\Big|  \{t \in (0,T) : 
|| u_n(t) - u^* (t)   ||_{X} \geq \ep   \}  \Big| =0 \quad \forall \ep>0.
$$
\end{remark}

Finally, we consider the classical trace operator $\mathcal{T}$.  Let $1 \leq p < \infty$ and suppose that $\Om$ is a bounded domain with Lipschitz boundary. Then for all $\Psi \in C^1(\R^d, \R^d)$ and $f \in W^{1,p}(\Om)$, the integration by parts in Sobolev space \cite[Theorem 4.6]{EG} is given as
\begin{equation}\label{byparts}
\int_{\Om} f~ \na \cdot \Psi dx = -\int_{\Om} \na f \cdot \Psi dx + \int_{\p \Om} (\Psi \cdot \nu) \mathcal{T}(f) d\mathcal{H}^{d-1}.
\end{equation}
Note that $\mathcal{T}$ has the following local properties  \cite[Theorem 5.7]{EG}: for $\mathcal{H}^{d-1}$-a.e. point $x \in \p \Om$,
\begin{equation*}
    \lim_{r \ra 0} \dashint_{B(x,r) \cap \Om} |f - \mathcal{T}f(x) | dy=0 	\quad and 
\quad \mathcal{T}f(x) = \lim_{r \ra 0} \dashint_{B(x,r) \cap \Om} f dy.
\end{equation*}

\vspace{3mm}

\section{Auxiliary Lemmas}\label{sec:lemmas}
\subsection{Well-definedness of JKO scheme}
We begin with the following lemma, which is obtained by estimating the energy functional $\calE$. Although the proof is simple, it is used frequently throughout this paper.
\begin{lemma}\label{young}
Let $\mu = \rho \mathcal{L}^d |_{\Om} \in \mathcal{M}_2(\Om)$ with  ${\calE}(\mu)< \infty$. Then there exists a positive constant $m_\infty$ depending on $\al, \lam, \Om, {\calE}(\mu)$ such that
\begin{equation}\label{bound2}
\int_{\Om} \rho^\al dx \leq m_\infty \quad \textrm{and} \quad \quad \int_{\Om} \rho ~dx \leq m_\infty.
\end{equation}
\end{lemma} 
\begin{proof}
By the definition of $\calE$, we have $\int_\Om \rho^\al =(\al-1)({\calE}(\mu) - |\Om| \lam^\al) + \al \lam^{\al-1} \int_\Om \rho$. Using Young's inequality, we get $\int_\Om \rho \leq \ep \int_\Om \rho^\al + C(\ep, \al, \Om)$. By choosing  $\ep>0$ sufficiently small, we see the first inequality in \eqref{bound2} holds. The second inequality is then obtained by H\"{o}lder's inequality.
\end{proof}

The following lemma ensures that the scheme introduced in \eqref{argmin} is well defined. 

\begin{lemma}\label{minimizer}
Let $\mu_0 \in \mathcal{M}_2(\Om)$ with ${\calE}(\mu_0)<\infty$ and $\T>0$.  
Then there exists a unique minimizer $\mu_k^\T$  for each step in \eqref{argmin}.       
\end{lemma}

\begin{proof}
Since the energy $\calE$ is decreasing for each step from \eqref{bound1}, it is sufficient to show the existence and the uniqueness of the minimizer of the first step.

We first observe that $\calE$ is lower semicontinuous for $Wb_2$  convergence. Indeed, let $(\mu_n) \subset \mathcal{M}_2(\Om)$ with $Wb_2(\mu_n, \mu)  \ra 0$. Without loss of generality, we can assume $\liminf {\calE}(\mu_n)<\infty$, which enables us to find a subsequence $(\mu_{n'})$  such that 
$$
\liminf {\calE}(\mu_n) = \lim {\calE}(\mu_{n'}) \quad 
and \quad  \sup {\calE}(\mu_{n'})<\infty.
$$
Using Lemma \ref{young} and weak compactness of $L^\al$, there exists a subsequence $(\mu_{n''}) = (\rho_{n''} \mathcal{L}^d |_\Om) \subset \mathcal{M}_2(\Om)$ such that  $\rho_{n''} \wra \rho_*$ in $L^\al(\Om)$,  which implies by \eqref{Wb2_iff}
\begin{equation}\label{Wb2_convergence}
Wb_2(\mu_{n''}, \mu_*) \ra 0,
\end{equation}
where $\mu_*:=\rho_* \mathcal{L}^d |_{\Om} \in \mathcal{M}_2(\Om)$. Hence, $\mu = \mu_*$.
Due to the fact that for any $x,y \geq 0$
$$
U(x) - U(y) \geq \frac{\al}{\al-1}(x-y) (y^{\al-1} - \lam^{\al-1}), 
$$
we conclude that
\begin{equation}\nonumber
\liminf_{n \ra \infty} {\calE}(\mu_{n})  = \lim_{n'' \ra \infty} {\calE}(\mu_{n''}) \geq {\calE}(\mu).
\end{equation}

Next, we prove the existence of a minimizer. Let $\mathcal{F}(\nu):=\calE(\nu ) + Wb_2^2(\mu_0, \nu)/2\T$. Clearly, $\mathcal{F} \geq0$. Since $\mathcal{F}$ is bounded below, there exists a minimizing sequence $(\nu_n) \subset \mathcal{M}_2(\Om)$ satisfying 
$$\textrm{inf}_{\mu \in \mathcal{M}_{2}(\Om)} \mathcal{F}(\mu)=\lim_{n\ra \infty} \mathcal{F}(\nu_n).$$
Using that $\textrm{inf}_{\mu \in \mathcal{M}_{2}(\Om)} \mathcal{F}(\mu)\leq \mathcal{F}(\mu_0) = {\calE}(\mu_0)$, we deduce that
$\nu_n=f_n \mathcal{L}^d |_{\Om}$ and $(f_n)$ is uniformly bounded in ${L^\al(\Om)}$ by Lemma \ref{young}.  Similarly as in \eqref{Wb2_convergence}, it follows that $Wb_2(\nu_{n_k}, \nu_{*}) \ra 0$ for a subsequence $(\nu_{n_k})$ of $(\nu_n)$.
Together with the lower semicontinuity of $\calE$, we obtain
$$
\textrm{inf}_{\mu \in \mathcal{M}_{2}(\Om)} \mathcal{F}(\mu)= \liminf_{k \ra \infty} \mathcal{F}(\nu_{n_k}) \geq 
\liminf_{k \ra \infty} \calE(\nu_{n_k}) + \liminf_{k \ra \infty} Wb_2^2(\mu_0, \nu_{n_k})/2\T \geq
\mathcal{F}(\nu_*).
$$
Therefore $\nu_*$ is a minimizer, namely $\mu_1^\T$.

Lastly, we prove that the minimizer is unique. Suppose $\mu^\T$ and $\tilde{\mu}^\T$ are two minimizers.
Define $\sigma_t^\T=(1-t)\mu^\T + t \tilde{\mu}^\T$, $0\leq t \leq 1$. Let us observe that the term $Wb_2^2(\mu_0, \cdot)$ is convex, namely:
\begin{equation}\label{min_ineq1}
Wb_2^2(\mu_0,\sigma_t^\T) \leq (1-t)Wb_2^2(\mu_0 ,\mu^\T) + tWb_2^2(\mu_0, \tilde{\mu}^\T).
\end{equation}
Indeed, we consider two optimal plans $\gam^\T \in OPT(\mu_0, \mu^\T)$ and $\tilde{\gam}^\T \in OPT(\mu_0, \tilde{\mu}^\T)$, and define $\gam_t^\T:= (1-t)\gam^\T +t \tilde{\gam}^\T \in ADM(\mu_0, \sigma_t^\T)$. 
Since $Wb_2^2(\mu_0, \sigma_t^\T) \leq \int |x-y|^2 d\gam_t^\T$, \eqref{min_ineq1} holds. 
In addition, using the minimality of $\mu^\T$ and \eqref{min_ineq1}, we have
$$
{\calE}(\mu^\T) + \frac{Wb_2^2(\mu_0, \mu^\T)}{2\T} \leq \calE(\sigma_t^\T) + \frac{Wb_2^2(\mu_0, \sigma_t^\T)}{2\T}
$$
\begin{equation}\nonumber
\leq  \calE(\sigma_t^\T) + \frac{(1-t)Wb_2^2(\mu_0 ,\mu^\T) + tWb_2^2(\mu_0, \tilde{\mu}^\T)}{2\T}.
\end{equation}
This implies
\begin{equation}\nonumber
\lim_{t \ra 0+} \frac{{\calE}(\mu^\T)  - \calE(\sigma_t^\T)}{t} =\int_{\Om} U'(\rho^\T)( \rho^\T - \tilde{\rho}^\T) dx \leq \frac{Wb_2^2(\mu_0, \tilde{\mu}^\T) - Wb_2^2(\mu_0, \mu^\T) }{2\T}.
\end{equation}
Similarly, the minimality of $\tilde{\mu}^\T$ and \eqref{min_ineq1} yield
\begin{equation}\nonumber
\int_{\Om} U'(\tilde{\rho}^\T)( \tilde{\rho}^\T - \rho^\T) dx \leq \frac{Wb_2^2(\mu_0, \mu^\T) - Wb_2^2(\mu_0, \tilde{\mu}^\T) }{2\T}.
\end{equation}
By adding the last two inequalities, we get
\begin{equation}\label{give_uniqueness}
\int_{\Om}(\tilde{\rho}^\T - \rho^\T)    \big(U'(\tilde{\rho}^\T)-  U'(\rho^\T) \big)dx \leq 0.
\end{equation}
The convexity of $U$ and \eqref{give_uniqueness}
imply $\mu^\T= \tilde{\mu}^\T$.
\end{proof}

\subsection{Euler--Lagrange equation}

The next proposition proven in \cite{FG}
 is needed to prove Lemma \ref{EL}.  
\begin{proposition}\cite[Proposition 2.11]{FG}\label{perturb1}
Let $\mu,\nu \in \mathcal{M}_2(\Om)$ and $\Phi : \Om \ra \R^d$ a bounded vector field which is compactly supported. Also, let $\gam \in OPT(\mu,\nu)$, and define $\mu_t:=(id+t \Phi)_{\#}\mu$. Then
$$
\limsup_{t \ra 0} \frac{Wb_2^2(\mu_t, \nu) - Wb_2^2(\mu,\nu)}{t} \leq -2 \int_{\OM \times \OM} \lag \Phi(x), y-x \rag d\gam(x,y).
$$
\end{proposition}
Then we get the following Euler-Lagrange equation for $\rho^\T$:
\begin{lemma}\label{EL}
Let $\mu^\T$ be a minimizer in \eqref{argmin}. If  $\gam \in OPT(\mu^\T, \mu_0)$ and $\Phi \in C_c^1(\Om, \R^d)$, then $\mu^\T = \rho^\T \mathcal{L}^d|_{\Om}$ and
\begin{equation}\nonumber
\int_{\OM \times \OM} \lag  \Phi(x), y-x \rag d\gam(x,y) + \T \int_{\Om} (\rho^{\T})^\al \emph{div} \Phi(x) dx = 0.
\end{equation}
\end{lemma}
\begin{proof}
For sufficiently small $t>0$, let us define $\mu_t^{\T}:= (id+t\Phi)_{\#}\mu^\T$. By minimality of $\mu^\T$, we have
$$
\frac{Wb_2^2(\mu^\T,\mu_0)}{2\T} + {\calE}(\mu^\T) \leq \frac{Wb_2^2(\mu_t^{\T},\mu_0)}{2\T} + {\calE}(\mu_t^{\T}).
$$
Then, by Proposition \ref{perturb1} (letting $\nu=\mu_0)$ and Proposition \ref{perturb2}, we obtain
\begin{equation}\nonumber
0 \leq -\frac{1}{\T} \int_{\OM \times \OM} \lag \Phi(x), y-x \rag d\gam(x,y) - \int_{\Om} (\rho^{\T})^\al \textrm{div} \Phi(x) dx.
\end{equation}
Replacing $\Phi$ with $-\Phi$, we get the desired result.
\end{proof}

\subsection{Fundamental estimates}  
Before proceeding to the main arguments, we first recall the following well-known and useful inequalities.

Let $\mu_0 \in \mathcal{M}_2(\Om)$ with ${\calE}(\mu_0)<\infty$. Due to the definition of the energy $\calE$, we know that $\mu^{\T} = \rho^{\T} \mathcal{L}^d |_{\Om}$ whenever the minimizer in \eqref{argmin} exists. We often use $\rho$ to denote the measure $\mu=\rho \mathcal{L}^d |_{\Om}$ by abuse of notation.

It can be easily shown that
\begin{equation}\label{bound1}
\sum_{k=m}^{n-1} \frac{Wb_2^2 (\rho_k^\T, \rho_{k+1}^\T)}{2\T} \leq \calE(\rho_m^\T) - \calE(\rho_{n}^\T) \leq \calE(\rho_0) \quad \forall  m <n \in \N
\end{equation}
by an inductive argument. 
The standard slope estimate \cite[Lemma~3.1.3]{AGS} gives, at every step,
\begin{equation}\label{energy_slope}
|\na\calE|(\rho_{k+1}^\T)
\leq
\frac{Wb_2(\rho_k^\T,\rho_{k+1}^\T)}{\T}
\qquad \forall k\geq0.
\end{equation}
Here $|\na\calE|$ is the slope defined in \eqref{def_of_slope}. 
Combining \eqref{energy_slope} with \eqref{bound1}, we obtain
\begin{equation}\label{Ambrosio}
\T\sum_{k=m}^{n-1}|\na\calE|^2(\rho_{k+1}^\T)
\leq
\sum_{k=m}^{n-1}\frac{Wb_2^2(\rho_k^\T,\rho_{k+1}^\T)}{\T}
\leq
2\bigl(\calE(\rho_m^\T)-\calE(\rho_n^\T)\bigr)
\leq 2\calE(\rho_0)
\quad \forall m<n.
\end{equation}

\subsubsection*{Uniform $L^1\cap L^\alpha$ estimates} In particular, since ${\calE}(\mu^\T (t)) \leq {\calE}(\mu_0)$, Lemma \ref{young} implies that there is a positive constant $m_\infty$, independent of $\T$, such that 
\begin{equation}\label{bound2_time}
||\rho^\T(t)||_{L^\al(\Om)}^\al \leq m_\infty, 
\quad
|| \rho^\T(t)||_{L^1(\Om)} \leq m_\infty \quad \forall t\geq0.
\end{equation}

\subsubsection*{Approximate Hölder continuity}
On the other hand, from \eqref{DS} and \eqref{bound1}, we directly obtain the approximate $\frac{1}{2}$-H\"{o}lder continuity:
\begin{equation}\label{HolderConti}
Wb_2(\rho^\T(s), \rho^\T(t)) \leq \sqrt{2\calE(\rho_0) ( |s-t| +  \T)} \quad \forall s,t \ge0.
\end{equation}

\vspace{3mm}
If $\z$ is Lipschitz and compactly supported in $\Om$, then we can estimate the difference between the integral of $\z$ against $\mu$ and the integral of $\z$ against $\nu$:
\begin{lemma}\label{munu}
Let $\mu, \nu \in \mathcal{M}_2(\Om)$. For a Lipschitz function $\z$ compactly supported in $\Om$, we have
$$
\Big|  \int_\Om \z d\mu - \int_\Om \z d\nu \Big| 
\leq Lip(\z) \sqrt{\mu(\Om) + \nu(\Om)} Wb_2(\mu,\nu).
$$
\end{lemma}
\begin{proof}
Let $\gam \in \textrm{OPT}(\mu,\nu)$. Then we see that
$$
\Big|  \int_\Om \z d\mu - \int_\Om \z d\nu \Big| = \Big| \int_{\Om \times \OM} \z(x) d\gam(x,y) - \int_{\OM \times \Om} \z(y) d\gam(x,y) \Big|
$$
$$
= \Big| \int_{\OM \times \OM} \z(x) d\gam(x,y) - \int_{\OM \times \OM} \z(y) d\gam(x,y) \Big| \leq \int_{\OM \times \OM} \big|  \z(x) -\z(y)  \big| d\gam(x,y)
$$
$$
\leq Lip(\z) \sqrt{\gam( \OM \times \OM)} Wb_2(\mu,\nu)  \leq Lip (\z) \sqrt{\mu(\Om) + \nu(\Om)} Wb_2(\mu,\nu)
$$
since $\gam (\OM \times \OM) \leq \gam(\Om \times \OM) + \gam(\OM \times \Om) = \mu(\Om) + \nu(\Om)$.
\end{proof}

In the next proposition, we evaluate the directional derivative of $\calE$ along $\Phi$:
\begin{proposition}\label{perturb2}
Suppose $\mu=\rho \mathcal{L}^d |_{\Om} \in \mathcal{M}_2(\Om)$ with ${\calE}(\mu)<\infty$, and let $\Phi \in C_c^1(\Om, \R^d)$. If $\mu_t:= (id + t \Phi)_{\#} \mu$, then
\begin{equation}\nonumber
\lim_{t \ra 0} \frac{{\calE}(\mu_t) - {\calE}(\mu)}{t}= -\int_{\Om} \rho^{\al} ~ \emph{div} \Phi dx.
\end{equation}
\end{proposition}
\begin{proof} 
For sufficiently small $t$, we know $\mu_t=\rho_t \mathcal{L}^d |_{\Om}\in \mathcal{M}_2(\Om)$ and $\int_\Om \rho_t  = \int_\Om \rho$. The rest of the proof is the same as the proof known in Wasserstein space (see, e.g., \cite[Lemma 10.4.4]{AGS}).
\end{proof}

We now turn to obtain a lower bound of $|\na \calE|$, which plays an important role when using the Aubin-Lions Lemma in section \ref{sec:thm1}, see \eqref{H^1}.

\begin{lemma}\label{slope}
Let $\mu =\rho \mathcal{L}^d |_{\Om} \in \mathcal{M}_2(\Om)$ with ${\calE}(\mu)<\infty$. If $|\na \calE|(\mu)<\infty$, then $\rho^\al \in W^{1,1}(\Om)$ and
\begin{equation}\nonumber
\int_{\Om} \frac{|\na \rho^\al |^2}{\rho} dx \leq |\na \calE |^2 (\mu).
\end{equation}
\end{lemma}

\begin{proof}
We follow the arguments of \cite[Theorem 10.4.6]{AGS} about $(\mathcal{P}_p(\R^d), W_p)$. 
Let  $\mu_t := (id+ t \Phi)_\# \mu$ with $\Phi \in C_c^1(\Om, \R^d)$. Using the definition of $|\na \calE|$, we see that
$$
|\na \calE|(\mu)
\geq  \limsup_{t \da 0} \frac{\Big({\calE}(\mu) -{\calE}(\mu_t) \Big)^+}{Wb_2(\mu,\mu_t)}
\geq \limsup_{ t \da 0} \frac{\Big({\calE}(\mu) -{\calE}(\mu_t) \Big)^+}{W_2(\mu,\mu_t)} 
$$
\begin{equation}\label{slope_lower_bound}
\geq \limsup_{ t \da 0} \frac{\Big({\calE}(\mu) -{\calE}(\mu_t) \Big)^+}{t} \frac{t}{W_2(\mu,\mu_t)}
\geq \limsup_{ t \da 0} \frac{\Big({\calE}(\mu) -{\calE}(\mu_t) \Big)^+}{t} \frac{1}{||\Phi||_{L^2(\mu; \Om)}},  
\end{equation}
where we used that $\mu(\Om)=\mu_t(\Om)$ for sufficiently small $t>0$ and that $W_2(\mu, \mu_t) \leq t ||\Phi||_{L^2(\mu ; \Om)}$. Using Proposition \ref{perturb2} and \eqref{slope_lower_bound},  we get
$$
\int_\Om \rho^\al  \textrm{div} \Phi dx
\leq |\na \calE|(\mu) ||\Phi||_{L^2(\mu;\Om)}.
$$
Changing $\Phi$ with $-\Phi$ gives
\begin{equation}\label{BV}
\Big|  \int_\Om \rho^\al  \textrm{div} \Phi dx \Big|
\leq |\na \calE|(\mu) ||\Phi||_{L^2(\mu;\Om)} 
\leq  \sqrt{\mu(\Om)} |\na \calE|(\mu) \sup_{\Om} |\Phi|.
\end{equation}
From Lemma \ref{young}, $\rho^\al \in L^1(\Om)$. Hence \eqref{BV} directly shows $\rho^\al \in BV(\Om)$ (see \cite[Chapter 3]{AFP} or \cite[Chapter 5]{EG}). Using Riesz's Representation theorem, we have a finite vector valued Radon measure, namely $D \rho^\al$, such that
\begin{equation}\label{BVparts}
-\int_{\Om} \rho^\al \textrm{div} \Phi dx =  \int_{\Om} \Phi  d (D\rho^\al ), \quad \forall \Phi \in C_c^1(\Om, \R^d).
\end{equation}
On the other hand, if we define
$$
\mathscr{L}(\Phi) := \int_{\Om} \Phi  d (D\rho^\al ) \quad  \textrm{and} \quad p(\Phi):= |\na \calE|(\mu) ||\Phi||_{L^2(\mu;\Om)},
$$
then $\mathscr{L}$ is a linear functional on  $C_c^1(\Om, \R^d) $ and $p$ is a sub-linear function on $L^2(\mu;\Om, \R^d)$. 

Using \eqref{BV}, \eqref{BVparts}, and Hahn-Banach theorem, there exists a linear functional 
$\widetilde{\mathscr{L}}$ such that 
$\widetilde{\mathscr{L}} =\mathscr{L}$ on $C_c^1(\Om, \R^d)$ and
$$
|\widetilde{\mathscr{L}} (\Phi)|
\leq |\na \calE|(\mu) ||\Phi||_{L^2(\mu;\Om)}, \quad  \forall \Phi \in L^2(\mu;\Om),
$$
which implies
$\widetilde{\mathscr{L}}$ is a bounded linear functional on $L^2(\mu;\Om)$. Therefore, there exists  $\bold{w} \in L^2(\mu;\Om)$ such that $||\widetilde{\mathscr{L}}||=||\bold{w}||_{L^2(\mu;\Om)} \leq |\na \calE|(\mu)$ and
\begin{equation}\label{from_riesz}
\widetilde{\mathscr{L}}(\Phi) = \int_{\Om} \Phi \cdot \bold{w} d\mu
\end{equation}
due to Riesz's representation theorem. By using Lemma \ref{young} and H\"{o}lder's inequality, we have $\int_\Om \sqrt{\rho} (\sqrt{\rho} |\bold{w}|) dx \leq m_\infty^{1/2} || \bold{w}||_{L^2(\mu;\Om)}$, and thus we have $\rho \bold{w} \in L^1(\Om)$. Observing \eqref{BVparts} and \eqref{from_riesz}, we have for all $\Phi \in C_c^{\infty}(\Om, \R^d)$
$$
\int_{\Om} \rho^\al \textrm{div}\Phi dx = -\int_{\Om} \Phi \cdot (\rho \bold{w}) dx,
$$
which yields $\rho^\al \in W^{1,1}(\Om)$, $\rho \bold{w}=\na \rho^\al$, and 
$$
\int_{\Om} \frac{|\na \rho^\al|^2}{\rho} dx \leq |\na \calE|^2(\mu).
$$
This completes the proof.
\end{proof}


\section{Analysis near the boundary}\label{sec:bdry}
In this section, we investigate the properties of the discrete solutions near the boundary. It suffices to consider the minimizer at the first step, as the same results follow at each step by identical arguments.

Let $\rho^\T$ be the minimizer determined by the first step of \eqref{argmin}. Let us define 
\begin{equation}\label{S_r_Z}
S_r:= \{ x \in \Om : d(x, \p \Om) < r \} \quad and  
\quad Z:= \{ x \in \Om:  \rho^\T(x)  =0 \}.
\end{equation}
Clearly, $S_r$ is well-defined for all $r>0$ (there exists a positive constant $r_ \Om \leq \textrm{diam}(\Om)$ such that $S_r = \Om$ if $r \geq r_\Om$).

We will prove that $\rho^\T$ cannot be zero near $\p \Om$ for $\lam>0$ in the next lemma, which will be crucially used to obtain the trace of $\rho^\T$ later.
\begin{lemma}\label{bdry_diffusion}
Set $r^*:=\displaystyle{\sqrt{\frac{2\al \lam^{\al-1}}{\al-1}\T} }$. If $\lam>0$, then
$
\mathcal{L}^d (S_r \cap Z)=0
$
for any $r \in (0,r^*)$, where $S_r$ and $Z$ are the sets defined in \eqref{S_r_Z}.
\end{lemma}

\begin{proof}
We prove the lemma by a contradictory argument. We assume that there exists $r_0 \in (0, r^*)$ such that 
$
\mathcal{L}^d (S_{r_0} \cap Z)>0.
$
For a given $z \in \Om$, for convenience, we denote by $P(z)$ a point that is one of the nearest points on $\p \Om$ with respect to $z$ (it is well known that the map $z \in \Om \mapsto P(z)$ is uniquely defined on $\mathcal{L}^d$-a.e.).

Let $\gam \in OPT(\rho^0, \rho^\T)$ and set
$$
\gam_{\ep}:=\gam + \ep (P,id)_{\#} \sigma \quad \textrm{with} \quad
\sigma =  \chi_{S_{r_0} \cap Z} dx.
$$
Then we define $\rho_{\ep}^\T:=(\pi_{\#}^2\gam_{\ep}) |_{\Om}=\rho^\T + \ep \sigma$ and note that $\gam_{\ep} \in ADM(\rho^0, \rho_{\ep}^\T)$.

From the minimality of $\rho^\T$ as in \eqref{argmin}, we get
\begin{equation}\label{diffusion_minimality}
\calE(\rho^\T) + \frac{Wb_2^2(\rho^0, \rho^\T)}{2\T}
\leq
\calE(\rho_{\ep}^\T) + \frac{Wb_2^2(\rho^0, \rho_{\ep}^\T)}{2\T}.
\end{equation}

Observe that
\begin{equation}\label{diffusion_E}
\calE(\rho^\T) - \calE(\rho_{\ep}^\T) 
=  \int_{S_{r_0} \cap Z} + \int_{\Om \backslash (S_{r_0} \cap Z)} 
= \int_{S_{r_0} \cap Z} U (\rho^\T) - U \big(\rho^\T + \ep  \big) 
= \int_{S_{r_0} \cap Z} U(0) - U(\ep)
\end{equation}
and
$$
Wb_2^2(\rho^0, \rho_{\ep}^\T) - Wb_2^2(\rho^0, \rho^\T) 
\leq \int_{\OM \times \OM} |x-y|^2 d\gam^\ep(x,y) - \int_{\OM \times \OM} |x-y|^2 d\gam(x,y) 
$$
\begin{equation}\label{diffusion_Wb2}
=\ep  \int_{S_{r_0} \cap Z} |P(x) -x|^2  dx.
\end{equation}
Combining \eqref{diffusion_minimality}, \eqref{diffusion_E}, and \eqref{diffusion_Wb2}, we have
$$
\int_{S_{r_0} \cap Z} U(0) - U(\ep) \leq 
\frac{\ep}{2\T}  \int_{S_{r_0} \cap Z} |P(x) -x|^2  dx.
$$
Dividing by $\ep$ and letting $\ep \downarrow 0$, we see that
\begin{equation}\nonumber
\frac{\al \lam^{\al-1}}{\al-1} \mathcal{L}^d( S_{r_0} \cap Z)
\leq
\frac{1}{2\T} \int_{S_{r_0} \cap Z} |P(x) -x|^2 dx \leq \frac{r_0^2}{2\T} \mathcal{L}^d (S_{r_0} \cap Z).
\end{equation}
Since we assumed $\mathcal{L}^d (S_{r_0} \cap Z)>0$, we get
$r^*
\leq r_0,
$
which leads to a contradiction.
\end{proof}

\vspace{3mm}
Thanks to Lemma \ref{bdry_diffusion}, we can determine the trace of discrete solutions by using ideas of the heuristic argument of \cite[Proposition 3.7]{FG}.

\begin{proposition}\label{DS_regularity}
Let $\rho^\T$ be a minimizer in \eqref{argmin}. We have the following:
\begin{enumerate}
\item[\emph{(i)}] 
(Case $\lam=0$) For almost every $x \in \Om$, 
\begin{equation}\label{key_inequality1}
0 \leq
\rho^\T (x) 
\leq \Big( \frac{3(\al-1) \emph{diam}(\Om)}{2\al \T} d(x,\p \Om) \Big)^{\frac{1}{\al-1}}.
\end{equation}

\item[\emph{(ii)}]  
(Case $\lam>0$) Let $0<\del< r^*= \sqrt{\frac{2\al \lam^{\al-1}}{\al-1}\T}$. Then for almost every $x \in S_\del$,

\begin{equation}\label{key_inequality2}
\Big( \lam^{\al-1} - \frac{(\al-1)d^2(x,\p \Om)}{2\al \T}\Big)^{\frac{1}{\al-1}} \leq \rho^\T (x) 
\leq \Big(\lam^{\al-1} + \frac{3(\al-1) \emph{diam}(\Om)}{2\al \T} d(x,\p \Om) \Big)^{\frac{1}{\al-1}}.
\end{equation}

\end{enumerate}
As a result, for all $\lam \geq 0$
\begin{equation}\label{sobolev}
(\rho^{\T})^{\al - \frac{1}{2}} - \lam^{\al - \frac{1}{2}} \in H_0^1(\Om).
\end{equation}
\end{proposition}

\begin{proof}
(i) 
Let $\gam \in OPT(\mu^0, \mu^\T)$. Then, by a variant of Brenier's theorem (see \cite[Corollary 2.5]{FG}), there exists a map $S:\Om \ra \OM$, which is the gradient of a convex function, satisfying $\gam|_{\OM \times \Om}=(S, id)_\# \mu^\T$. Fix $z \in \Om$. To prove \eqref{key_inequality1}, without loss of generality, we can assume 
$\rho^\T(z)>0$.

For small $r,\ep>0$, if we define
$$
\gam_{r, \ep} : = \gam - \ep \gam|_{\OM \times B_r(z)}+ \ep (id, P)_\# \gam|_{\OM \times B_r(z)}
$$
and
$$
  (\pi_2)_{\#}\gam_{r,\ep} |_{\Om}:=\mu_{r,\ep}^\T=\rho_{r,\ep}^\T \mathcal{L}^d |_{\Om}, 
$$
then $\gam_{r,\ep} \in ADM(\mu^0 , \mu_{r,\ep}^\T)$ and
\begin{equation}\nonumber 
   \rho_{r,\ep}^\T(y)=
  \begin{cases}
                                   \rho^\T(y) & \text{if  $ y \in \big( B_r(z) \big)^c $}, \\
                                   
                                   (1-\ep)\rho^\T(y) & \text{if  $y \in B_r(z)$}.  \\
  \end{cases}
\end{equation}
Since $\mu_{r,\ep}^\T \in \mathcal{M}_2(\Om)$, by minimality of $\mu^\T$, we have
\begin{equation}\label{minimality}
{\calE}(\mu^\T) + \frac{Wb_2^2(\mu^0, \mu^\T)}{2\T} \leq {\calE}(\mu_{r, \ep}^\T)+ \frac{Wb_2^2(\mu^0, \mu_{r,\ep}^\T)}{2\T},
\end{equation}
which gives
$$
\int_{\Om } U(\rho^\T (y))dy+\frac{1}{2\T} \int_{\OM \times \OM} |x-y|^2 d \gam(x,y)
\leq \int_{\Om} U(\rho_{r,\ep}^\T(y))dy + \frac{1}{2\T} \int_{\OM \times \OM} |x-y|^2 d \gam_{r,\ep}(x,y).
$$
Note that $\gam_{r,\ep}|_{\OM \times \Om}=(S,id)_\# \mu_{r,\ep}^\T$.
Splitting $\int_{\OM \times \OM} \cdot = \int_{\OM \times \Om} \cdot + \int_{\OM \times \p \Om} \cdot$,  we see that
$$
\int_{B_r(z)} U(\rho^\T)dy + \frac{1}{2\T} \int_{\Om} |y-S(y)|^2 \rho^\T(y) dy
$$

$$
\leq \int_{B_r(z)}  U\big(  (1-\ep)\rho^\T(y) \big) dy + \frac{1}{2\T} \int_{\Om}|y-S(y)|^2 \rho_{r,\ep}^{\T}(y)dy + \frac{\ep}{2\T} \int_{\OM \times B_r(z)} |x-P(y)|^2 d\gam(x,y).
$$
Dividing both sides by $\ep$ and letting $\ep \ra 0$, we have
$$
\int_{B_r(z)} U'(\rho^\T(y))\rho^\T(y)dy + \frac{1}{2\T} \int_{B_r(z)} |y-S(y)|^2 \rho^\T(y)dy
$$
$$
\leq \frac{1}{2\T}  \int_{\OM \times B_r(z)} |x-P(y)|^2 d\gam(x,y)=  \frac{1}{2\T}\int_{B_r(z)} |S(y)- P(y)|^2 \rho^\T(y) dy,
$$
where we used that $\gam|_{\OM \times  B_r(z)} = (S,id)_{\#} (\mu^\T |_{B_r(z)})$.

\noindent Using the inequality
$|S(y)-P(y)|^2 \leq |y-S(y)|^2 +  3\textrm{diam}(\Om) |P(y)-y|$, 
we easily get
$$
\int_{B_r(z)} U'(\rho^\T(y))\rho^\T(y)dy
\leq
\frac{3\textrm{diam}(\Om)}{2\T}\int_{B_r(z)} d(y, \p \Om) \rho^\T(y) dy.
$$
Since $\rho^\T(z)>0$ and $U'(\rho^\T)\rho^\T \in L^1(\Om)$,
dividing both sides by $\mathcal{L}^d (B_r(z))$ and letting $r \ra 0$, we finally obtain
\begin{equation*}
U'(\rho^\T(z)) \leq \frac{3 \textrm{diam}(\Om)}{2\T} d(z,\p \Om)  \quad \mathcal{L}^d \textrm{-}a.e. ~z \in \Om,
\end{equation*}
which proves \eqref{key_inequality1}.

(ii) By Lemma \ref{bdry_diffusion}, we know that $\rho^\T>0$ a.e. in $S_\del$.
Hence, the second inequality in \eqref{key_inequality2} can be proved by repeating similar arguments in (i).

To prove the first inequality in \eqref{key_inequality2}, we choose $r>0$ small enough so that $B_r(z) \subset S_\del$ and follow the argument above  by considering the other perturbation,
$$
\tilde{\gam}_{r,\ep}:= \gam + \ep (P,id)_{\#}\rho^\T|_{B_r(z)}.
$$
Let us define
$$
  (\pi_2)_{\#}\tilde{\gam}_{r,\ep} |_{\Om}:=\tilde{\mu}_{r,\ep}^\T=\tilde{\rho}_{r,\ep}^\T \mathcal{L}^d |_{\Om}.
$$
Then $\tilde{\gam}_{r,\ep} \in ADM(\mu^0 , \tilde{\mu}_{r,\ep}^\T)$ and
\begin{equation}\nonumber
   \tilde{\rho}_{r,\ep}^\T(y)=
  \begin{cases}
                                   \rho^\T(y) & \text{if  $ y \in \big( B_r(z) \big)^c $}, \\
                                   
                                   (1+\ep)\rho^\T(y) & \text{if  $y \in B_r(z)$}.  \\
  \end{cases}
\end{equation}
Proceeding similarly to \eqref{minimality}, we obtain
$$
\int_{B_r(z)} U(\rho^\T)dy 
\leq \int_{B_r(z)}  U\big(  (1+\ep)\rho^\T(y) \big) dy +  \frac{\ep}{2\T} \int_{\OM \times \OM} |x-y|^2 d\Big( (P,id)_\# \rho^\T |_{B_r(z)}   \Big)(x,y).
$$
Dividing both sides by $\ep$ and letting $\ep \ra 0$, we get
$$
0 \leq \int_{B_r(z)} U'(\rho^\T(y))\rho^\T(y)dy +\frac{1}{2\T} \int_{B_r(z)} d^2(y, \p \Om) \rho^\T(y)dy
$$
because
$$
\int_{\OM \times \OM} |x-y|^2 d\Big( (P, id)_\# \rho^\T |_{B_r(z)}   \Big)(x,y)=\int_{B_r(z)} |P(y)-y|^2 \rho^\T(y)dy = \int_{B_r(z)} d^2(y, \p \Om) \rho^\T(y)dy.
$$
Since $\rho^\T(z)>0$ and  $U'(\rho^\T)\rho^\T \in L^1(\Om)$, dividing both sides by $\mathcal{L}^d (B_r(z))$ and letting $r \ra 0$, we finally get
\begin{equation}\nonumber
-\frac{d^2(z,\p \Om)}{2\T} \leq U'(\rho^\T(z)) \quad \mathcal{L}^d \textrm{-}a.e. ~z \in S_\del.
\end{equation}
We thus obtain \eqref{key_inequality2}.

It remains to show \eqref{sobolev}. 
From \eqref{Ambrosio}, we know that the slope of $\calE$ at $\rho^\T$ is finite. Hence, using Lemma \ref{slope} we have $\na (\rho^{\T})^{\al - \frac{1}{2}} \in L^2(\Om)$.
Since $(\rho^\T)^{\al-\frac{1}{2}} \in L^{\al/(\al-\frac{1}{2})}(\Om)$, by Gagliardo-Nirenberg interpolation inequality, we obtain
$$
(\rho^{\T})^{\al - \frac{1}{2}} - \lam^{\al - \frac{1}{2}} \in H^1(\Om).
$$
In addition, \eqref{key_inequality1} and \eqref{key_inequality2} imply that the trace of $(\rho^\T)^{\al - \frac{1}{2}}$ should be $\lam^{\al - \frac{1}{2}}$, and therefore we get the desired result.
 \end{proof}

\begin{remark}\label{trace}
Let $(\rho_k^\T)$ be the sequence of minimizers obtained from \eqref{argmin}. We repeatedly apply the above lemma to $\rho_k^\T$. Due to the existence of $S_\delta$ in Proposition \ref{DS_regularity} (ii), we can conclude that for a given $\rho_{0} \in \mathcal{M}_2(\Om)$ with $\calE(\rho_0) <\infty$, the function $t \mapsto (\rho^\T)^{\al - \frac{1}{2}}(t) -\lam^{\al - \frac{1}{2}}$ has zero trace for $t>0$.
Moreover, in the process of taking $\T \downarrow 0$ to determine the trace of its limit, the fact that $S_\delta$ shrinks as $\T \downarrow 0$ gives no problem (see the proof of (iii) in Theorem \ref{main}).
\end{remark}

\vspace{3mm}
\section{Proof of Theorem \ref{main}}\label{sec:thm1}
We are now ready to prove Theorem \ref{main}. Each subsection deals with the proof of (i)-(iv) in order.

\subsection{Convergence}
To prove (i), we divide the proof into two steps.\\
$\bullet$ \textit{Step 1. Convergence of $\rho^\T(t)$ in $\mathcal{M}_2(\Om)$}:
We recall that the set $\{ \rho^\T \}_{\T>0}$ belongs to $\mathcal{M}_{\leq m_{\infty}}$ by \eqref{bound2_time} and $\mathcal{M}_{\leq m_{\infty}}$ is compact in $\mathcal{M}_2(\Om)$ as mentioned in preliminaries. From \eqref{HolderConti}, we know the approximate $\frac{1}{2}$-H\"{o}lder continuity of $(\rho^\T)$.
With the aid of these facts, a refined version of the Ascoli-Arzel\`{a} theorem \cite[Proposition 3.3.1, Remark 2.1.1]{AGS} gives that for any $T>0$  there exists a vanishing sequence $(\T_k^{(1)})$ and $\rho^{(1)} \in  C([0,T], \mathcal{M}_2(\Om))$ such that for every $t \in [0,T]$,
\begin{equation}\nonumber
 Wb_2(\rho^{\T_k^{(1)}}(t) ,\rho^{(1)}(t)) \ra 0 \quad  as \quad k \ra \infty.
\end{equation}
We iterate this process on $[T,2T]$, $[2T,3T]$, and so on. By a diagonal argument, we can find a subsequence $(\T_k)$ and $\rho:[0,\infty) \ra \mathcal{M}_2(\Om)$ such that 
\begin{equation}\label{convergence1}
 Wb_2(\rho^{\T_k}(t) ,\rho(t)) \ra 0 \quad  as \quad k \ra \infty  \quad \forall t\geq0.
\end{equation}
Moreover, the weak lower semicontinuity of $Wb_2$ and \eqref{HolderConti} yield
\begin{equation}\label{HolderConti2}
Wb_2(\rho(t),\rho(s)) \leq \sqrt{2\calE(\rho_0)(t-s)}, \quad 0\leq s \leq t,
\end{equation}
that is, the curve $\rho$ is $\frac{1}{2}$-H\"{o}lder continuous with respect to $Wb_2$.

\noindent $\bullet$ \textit{Step 2. Convergence of $\rho^\T(t)$ in $L^\al \big( (0,T) \times \Om \big)$}:
Let $(\T_k)$ be the sequence obtained in \textit{Step 1}. We claim that for any $T>0$, $\rho^{\T_k}$ strongly converges, up to a subsequence, to $\rho$ in $L^\al \big( (0,T)\times \Om \big)$.
To get a priori estimate, let us observe that from Lemma \ref{slope} and \eqref{Ambrosio},  
$$
\int_0^T \int_{\Om} \Big|  \na \Big(  (\rho^{\T_k}(t))^{\al - \frac{1}{2} } \Big)   \Big|^2 dxdt = \frac{(\al-\frac{1}{2})^2}{\al^2} \int_0^T \int_{\Om} \frac{|\na (\rho^{\T_k}(t))^\al |^2}{\rho^{\T_k}(t) } dxdt
$$

\begin{equation}\label{H^1}
\leq C  \int_0^T |\na \calE|^2(\rho^{\T_k}(t))dt \leq C {\calE}(\mu_0).
\end{equation}
By Gagliardo-Nirenberg interpolation inequality, we  also have
\begin{equation}\nonumber
|| (\rho^{\T_k})^{\al- \frac{1}{2}}  ||_{L^2(\Om)} 
\leq C \Big(  ||\na (\rho^{\T_k})^{\al- \frac{1}{2}} ||_{L^2(\Om)}^\theta  m_{\infty}^{ \frac{(\al-\frac12)(1-\theta)}{\al} }+ m_{\infty}^{\frac{\al-\frac12}{\al}}         \Big), 
\quad \textrm{with} \quad
\theta:= \frac{
 \frac{\al-\frac{1}{2}}{\al}  - \frac{1}{2} 
}{ \frac{1}{d} +  \frac{\al-\frac{1}{2}}{\al}  - \frac{1}{2} },
\end{equation}
which implies, by Young's inequality,
\begin{equation}\label{most_important}
\int_0^T \int_{\Om} \Big|    (\rho^{\T_k}(t))^{\al - \frac{1}{2} }    \Big|^2 dxdt  
\leq C \Big(  \int_0^T \int_{\Om} \Big|  \na \Big(  (\rho^{\T_k}(t))^{\al - \frac{1}{2} } \Big)   \Big|^2 dxdt + 1  \Big).
\end{equation}
Combining \eqref{H^1},  \eqref{most_important},  and Remark \ref{trace}, we see that
\begin{equation}\label{H_0^1}
\textrm{the curves}~ t \mapsto  (\rho^{\T_k}(t))^{\al - \frac{1}{2} } - \lam^{\al-\frac{1}{2}}~ \textrm{are uniformly bounded in}~ L^2 \big( (0,T); H_0^1(\Om) \big). 
\end{equation}

In order to apply Proposition \ref{ALRS}, we define $X:=L^{\al}(\Om)$,  $g: X \times X \ra [0,\infty]$ given by
\begin{equation}\nonumber 
g(u, \tilde{u}):= 
  \begin{cases}
Wb_2(u, \tilde{u})      & \text{if $u,\tilde{u} \in \mathcal{M}_2(\Om)$ },             
           \\   
           \infty   & \text{  else,}
  \end{cases}
\end{equation}
and
\begin{equation}\label{functional_F}
	\mathscr{F}(u):=
  \begin{cases}
                           \displaystyle{   \int_{\Om} |\na(u^{\al - \frac{1}{2}} )|^2  dx  } & \text{if $u\in \mathcal{M}_2(\Om)$ and $u^{\al-\frac{1}{2}} - \lam^{\al-\frac{1}{2}} \in H_0^1(\Om)$}, \\
                                   
                                   \infty & \text{else}.  \\
  \end{cases}
\end{equation}
We now show that $\mathscr{F}$ and $g$ satisfy the assumptions. Here we adapt the strategy of proof of \cite[Proposition 4.8]{FM}, the main differences are that we define different $\mathscr{F}$, $g$ and use \eqref{H^1} to show \eqref{time_H^1}.

It can be easily checked that $g$ is lower semicontinuous on $L^\al(\Om) \times L^\al(\Om)$ by using the lower semicontinuity with respect to weak convergence in duality with functions in $C_c(\Om)$.

Let us define a sublevel of $\mathscr{F}$, that is, $\mathscr{F}_{\leq c}:= \{ u \in X: \mathscr{F}(u) \leq c  \}$. If $(u_n) \subset \mathscr{F}_{\leq c}$ then, by Rellich's theorem, there exists a subsequence $(n_k)$  such that
\begin{equation}\label{Rellich1}
u_{n_k}^{\al - \frac{1}{2}} \ra \sigma \quad \textrm{in} \quad L^2(\Om).
\end{equation}
Then the following inequality,
\begin{equation}\label{rudin}
\quad |x^p - y^p|\leq
  \begin{cases}
                           |x-y|^p & \text{if $0<p<1$}, \\ 
                                   p|x-y|(x^{p-1}+y^{p-1}) & \text{if $1\leq p < \infty$},  \\
  \end{cases}
\quad   \forall x,y\geq0,
\end{equation}
and \eqref{Rellich1} imply that for any $\al \geq 1$,
$$
u_{n_k} \ra \sigma^{1/(\al-\frac{1}{2})} \quad \textrm{in} \quad L^{2(\al-\frac{1}{2})}(\Om).
$$
Hence, $\mathscr{F}_{\leq c}$ is relatively compact in $L^{\al}(\Om)$ due to the fact that $2(\al-\frac{1}{2}) \geq \al$.

The lower semicontinuity of $\mathscr{F}$ on $L^\al(\Om)$ can be shown as follows:
Let $(u_n)$ be a sequence that converges to $u$ in $L^\al(\Om)$, with $\liminf_{n \ra \infty} \mathscr{F}(u_n) < \infty$. Clearly, there exists a subsequence $(n')$ such that
$$
\liminf_{n \ra \infty} \mathscr{F}(u_n)=\lim_{n' \ra \infty} \mathscr{F}(u_{n'}) \quad \textrm{and} \quad \sup_{n'} \mathscr{F}(u_{n'})<\infty.
$$  
For notational simplicity, without loss of generality, we 
shall write $(u_n)$ instead of $(u_{n'})$. Note that by \eqref{rudin}, we have
\begin{equation}\label{strong}
u_n^{\al-\frac{1}{2}} \ra u^{\al-\frac{1}{2}} \quad \textrm{in} \quad L^{\al/(\al-\frac{1}{2})} (\Om).
\end{equation}
Applying Poincaré's inequality on \eqref{functional_F}, we know $|| u_n^{\al-\frac{1}{2}} ||_{H^1(\Om)} \leq C$ for some $C>0$. 
By using the weak compactness and \eqref{strong}, one can extract a subsequence (still denoted $u_n$) such that
\begin{equation}\label{double_weak_L2}
u_n^{\al-\frac{1}{2}} \wra u^{\al-\frac{1}{2}} \quad \textrm{in} \quad L^2(\Om), \quad  
\na \big( u_n^{\al-\frac{1}{2}} \big) \wra \na u^{\al-\frac{1}{2}} \quad \textrm{in} \quad L^2(\Om).
\end{equation}
Since $u_n^{\al-\frac{1}{2} }$ has trace $\lam^{\al- \frac{1}{2}}$, \eqref{byparts} and \eqref{double_weak_L2} imply that $u^{\al-\frac{1}{2}}$ has also trace $\lam^{\al- \frac{1}{2}}$ (see also the proof of (iii) in Theorem \ref{main}). Therefore, $\mathscr{F}(u)<\infty$.
The elementary inequality $|x|^2-|y|^2 \geq 2(x-y)\cdot y$ gives
\begin{equation}\label{normal}
\mathscr{F}(u_n)- \mathscr{F}(u) \geq 2 \int_{\Om} \big( \na u_n^{\al-\frac{1}{2}} - \na u^{\al-\frac{1}{2}}     \big) \cdot \na u^{\al-\frac{1}{2}} dx.
\end{equation}
As $\na u_n^{\al-\frac{1}{2}} \wra \na u^{\al-\frac{1}{2}}$ in $L^2(\Om)$, taking $\lim_{n \ra \infty}$  in \eqref{normal} directly shows
\begin{equation*}
\liminf_{n\ra \infty} \mathscr{F}(u_n) \geq  \mathscr{F}(u),
\end{equation*}
which means that $\mathscr{F}$ is lower semicontinuous in $L^\al(\Om)$.

Now, set $U=\{ \rho^{\T_k} \}$. Then from \eqref{H_0^1}, we get
$$
\sup_{k} \int_0^T \mathscr{F}(\rho^{\T_k}(t)) dt <\infty.
$$
Next, if $h\geq \T_k$ then 
\begin{equation}\label{AL_limsup2}
\sup_{k} \int_0^{T-h} Wb_2(\rho^{\T_k}(t+h) , \rho^{\T_k}(t))dt \leq    (T-h)\sqrt{4{\calE}(\mu_0) h },
\end{equation}
where we used  \eqref{HolderConti}. When $0<h<\T_k$, choose $N$ for which $(N-1)\T_k < T \leq N\T_k$. Then we see
$$
\int_0^{T-h} Wb_2(\rho^{\T_k}(t+h),\rho^{\T_k}(t)) dt \leq h \sum_{n=0}^{N-1} Wb_2(\rho_n^{\T_k} , \rho_{n+1}^{\T_k})
\leq hN^{1/2} \Big( \sum_{n=0}^{N-1} Wb_2^2(\rho_n^{\T_k}, \rho_{n+1}^{\T_k}) \Big)^{1/2}
$$
\begin{equation}\label{AL_limsup3}
\leq h(2N\T_k)^{1/2}\Big( \sum_{n=0}^{N-1} \frac{Wb_2^2(\rho_n^{\T_k}, \rho_{n+1}^{\T_k})}{2\T_k}\Big)^{1/2} 
\leq h \big(  2(T+1){\calE}(\mu_0)  \big)^{1/2}.
\end{equation}
Thus,
\eqref{AL_limsup2} and \eqref{AL_limsup3} prove \eqref{AL_limsup}.

Therefore, by the conclusion of Proposition \ref{ALRS},  $\rho^{\T_k}$ converges in measure, up to a subsequence, to $\rho^*$ in the sense that
$$
\lim_{k \ra \infty} \Big| \{ t \in (0,T) : ||  \rho^{\T_k} (t) - \rho^* (t) ||_{L^\al(\Om)}  \geq \ep \}  \Big| = 0 \quad \forall \ep>0,
$$
which implies that there is a subsequence (still denoted $\T_k$) such that $\rho^{\T_k}$ converges in time a.e. to $\rho^*$ in $L^\al(\Om)$. For this $(\T_k)$, we know that \eqref{convergence1} still holds, which gives $\rho^* = \rho$ for a.e. because for any $\vphi \in C_c^{\infty}(\Om)$,
$$
\Big| \int_{\Om} (\rho(t) - \rho^*(t) ) \vphi dx \Big| \leq \Big| \int_{\Om} (\rho(t) - \rho^{\T_k}(t) ) \vphi dx \Big| + \Big| \int_{\Om} (\rho^{\T_k}(t) - \rho^*(t) ) \vphi dx \Big| \ra 0 \quad \textrm{as} \quad k \ra \infty.$$
Thus, for a.e. $t \in (0,T)$ we have 
\begin{equation}\label{convergence2}
\rho^{\T_k}(t) \ra \rho(t) \quad \textrm{in} \quad  L^\al(\Om).
\end{equation}
In fact, using \eqref{bound2_time} and dominated convergence theorem, we further see that 
\begin{equation}\label{convergence3}
\rho^{\T_k} \ra \rho \quad \textrm{in} \quad  L^\al \big( (0,T) \times \Om \big).
\end{equation}

\subsection{PME equation}
We now derive an approximation of weak solutions from $\rho^\T(t)$:
Given any $\T>0, ~k\in \N$, let $\gam_k  \in OPT( \rho_{k+1}^\T ,\rho_k^\T)$. For a given $\z \in C_c^{\infty}(\Om)$, we observe that
$$
\Big|\int_{\Om} \rho_k^\T(y) \z(y)dy -\int_{\Om} \rho_{k+1}^\T(x) \z(x)dx - \int_{\OM \times \OM} \na \z(x) \cdot(y-x) d\gam_k(x,y) \Big|
$$
$$
=\Big| \int_{\OM \times \Om} \z(y) d\gam_k(x,y) - \int_{\Om \times \OM} \z(x) d\gam_k(x,y) - \int_{\OM \times \OM} \na \z(x) \cdot (y-x) d\gam_k(x,y)\Big|
$$
$$
=\Big| \int_{\Om \times \Om} \z(y) - \z(x) -\na \z(x)\cdot(y-x) d\gam_k(x,y) + \int_{\p \Om \times \Om} \z(y) d\gam_k(x,y) 
$$

$$
- \int_{\Om \times \p \Om} \z(x) d\gam_k(x,y)  -  \int_{\Om \times \p \Om}  \na \z(x)\cdot (y-x) d\gam_k(x,y) \Big|
$$

$$
\leq \frac{1}{2} \sup_{\Om}|\na^2 \z| \int_{\Om \times \Om} |x-y|^2 d\gam_k(x,y) + \int_{\p \Om \times \Om} |\z(y)| d\gam_k(x,y)
$$

\begin{equation}\label{one_step_WS}
+\int_{\Om \times \p \Om} |\z(x)| d\gam_k(x,y) + \textrm{Lip}(\na \z) \int_{\Om \times \p \Om} |x-y|^2 d\gam_k(x,y).
\end{equation}

\noindent Note that
$$
\int_{\p \Om \times \Om} |\z(y)| d\gam_k(x,y) = \int_{\p \Om \times \rm{supp} (\z)} \frac{|\z(y)|}{|x-y|^2}|x-y|^2 d\gam_k(x,y)
\leq \frac{\sup_{\Om} |\z|}{c_{\z}^2}Wb_2^2(\rho_{k+1}^\T , \rho_k^{\T}),
$$
where $c_\z:=\min d(\p \Om, \textrm{supp}(\z))>0$.
Plugging this into \eqref{one_step_WS} yields
$$
\Big|\int_{\Om} \rho_k^\T(y) \z(y)dy -\int_{\Om} \rho_{k+1}^\T(x) \z(x)dx - \int_{\OM \times \OM} \na \z(x) \cdot(y-x) d\gam_k(x,y) \Big|
$$

$$
\leq \Big(\frac{1}{2} \sup_{\Om}|\na^2 \z |  + 2 \frac{\sup_{\Om} |\z|}{c_\z^2}+ \textrm{Lip}(\na \z)   \Big) Wb_2^2(\rho_{k+1}^\T, \rho_k^\T).
$$
Using Lemma \ref{EL} with $\Phi=\na \z$, we directly have
\begin{equation}\label{estimate1}
\Big|\int_{\Om} \rho_{k+1}^\T \z -\int_{\Om} \rho_k^\T \z - \T \int_{\Om} (\rho_{k+1}^\T)^{\al} \Delta \z \Big| \leq C ~ Wb_2^2(\rho_{k+1}^\T, \rho_k^\T).
\end{equation}

For given $0\leq t_1<t_2$, let $ \lceil \frac{t_1}{\T} \rceil+1=m$ and $\lceil \frac{t_2}{\T} \rceil=n$. 
Summing up from $m$ to $n$ and using \eqref{bound1}, \eqref{estimate1}, we obtain
$$
\Big|\int_{\Om} \rho_{n+1}^\T \z - \int_{\Om} \rho_m^\T \z - \sum_{k=m}^{n} \int_{k\T}^{(k+1)\T} \int_\Om (\rho_{k+1}^\T)^{\al} \Delta \z \Big| 
\leq C \sum_{k=m}^{n} Wb_2^2( \rho_{k+1}^\T, \rho_k^\T) \leq C \T.
$$

\vspace{3mm}

Using \eqref{HolderConti} and Lemma \ref{munu}, we get
\begin{equation}\label{E1}
\Big| \int_{\Om} \rho_{n+1}^\T \z - \int_{\Om} \rho^\T(t_2)\z \Big| 
\leq \Big| \int_{\Om} \rho_{n+1}^\T \z - \int_{\Om} \rho_n^\T\z \Big|  + \Big| \int_{\Om} \rho_{n}^\T \z - \int_{\Om} \rho^\T(t_2) \z \Big| 
\leq C\sqrt{\T}.
\end{equation}
Similarly,
\begin{equation}\label{E2}
\Big|  \int_{\Om} \rho^\T(t_1)\z - \int_{\Om} \rho_m^\T \z  \Big|  \leq C\sqrt{\T}.
\end{equation}
Since the function $t \mapsto \rho^\T(t)$ is piecewise-constant, we observe that
\begin{equation}\label{E3}
\Big| \sum_{k=m}^{n} \int_{k\T}^{(k+1)\T} \int_\Om (\rho_{k+1}^\T)^{\al} \Delta \z -\int_{t_1}^{t_2} \int_{\Om} (\rho^\T(r))^\al \Delta \z \Big| \leq  
4 m_\infty  ||\Delta \z||_{L^{\infty}(\Om)} \T.
\end{equation}

\noindent Combining \eqref{E1}-\eqref{E3}, we obtain
\begin{equation}\label{discreteWS}
\Big| \int_{\Om} \rho^\T (t_2) \z  - \int_{\Om} \rho^\T(t_1) \z - \int_{t_1}^{t_2} \int_{\Om} \big( \rho^\T(r) \big)^{\al} \Delta \z \Big| \leq C(\T + \sqrt{\T}),
\end{equation}
where $C$ depends on $\calE(\rho_0)$, $\z$, and $m_\infty$.

To prove \eqref{weak_solution}, for given $0 \leq t_1<t_2$, choose $T \geq t_2$ and set $\T=\T_k$ as obtained in (i).
Letting $\T_k \da 0$ in \eqref{discreteWS}, and using \eqref{convergence1} and \eqref{convergence3}, we conclude that 
\begin{equation}\nonumber
 \int_{\Om} \rho (t_2) \z dx  - \int_{\Om} \rho(t_1) \z dx= \int_{t_1}^{t_2} \int_{\Om} \big( \rho(r) \big)^{\al} \Delta \z ~dx dr, \quad 0\leq t_1 \leq t_2, \quad \forall \z \in C_c^\infty(\Om).
\end{equation}

\subsection{Boundary condition} 
To show \eqref{regularity}, let $(\T_k)$ be the sequence obtained in (i). Since \eqref{H_0^1} still holds for this $(\T_k)$, and since \eqref{convergence3} and \eqref{rudin} imply
$$
(\rho^{\T_k})^{\al-\frac{1}{2}} \ra \rho^{\al-\frac{1}{2}} \quad \textrm{in} \quad L^{\al/(\al-\frac{1}{2}) } \big((0,T) \times \Om \big),
$$
so we have, by weak compactness,
\begin{equation}\label{trace_convergence1}
(\rho^{\T_k})^{\al-\frac{1}{2}} 
\wra \rho^{\al-\frac{1}{2}}
 \quad \textrm{in} \quad L^2 \big((0,T) \times \Om \big)
\end{equation}
and
\begin{equation}\label{trace_convergence2}
\na (\rho^{\T_k})^{\al-\frac{1}{2}} 
\wra \na \rho^{\al-\frac{1}{2}}
 \quad \textrm{in} \quad L^2 \big((0,T) \times \Om \big)
\end{equation}
up to a subsequence (still denoted $(\T_k)$).
Therefore $\rho^{\al-\frac{1}{2}} \in L^2 \big((0,T); H^1(\Om) \big)$.

It remains to show that $\mathcal{T}(\rho^{\al - \frac{1}{2}}(t))=\lam^{\al - \frac{1}{2}}$ for almost every $t \in (0,T)$, where $\mathcal{T}$ is the classical trace operator.
Take any smooth and bounded vector field $\Phi$ on $\OM$. As $(\rho^{\T_k})^{\al - \frac{1}{2}}(t) \in H^1(\Om)$ and $\rho^{\al - \frac{1}{2}}(t) \in H^1(\Om)$ for a.e. $t \in (0,T)$, using \eqref{byparts} we know
$$
\int_{\Om} (\rho^{\T_k})^{\al- \frac{1}{2}}(t) ~ \na \cdot \Phi dx = -\int_{\Om} \na (\rho^{\T_k})^{\al- \frac{1}{2}}(t) \cdot \Phi dx + \int_{\p \Om} (\Phi \cdot \nu) \mathcal{T}\Big((\rho^{\T_k})^{\al- \frac{1}{2}}(t) \Big)	d\mathcal{H}^{d-1},
$$
$$
\int_{\Om} \rho^{\al- \frac{1}{2}}(t) ~ \na \cdot \Phi dx = -\int_{\Om} \na \rho^{\al- \frac{1}{2}}(t) \cdot \Phi dx + \int_{\p \Om} (\Phi \cdot \nu) \mathcal{T}\Big(\rho^{\al- \frac{1}{2}}(t) \Big) d\mathcal{H}^{d-1}.
$$
Due to \eqref{trace_convergence1} and \eqref{trace_convergence2}, we have for all $\eta \in C_c^{\infty}(0,T)$
$$
\int_0^T \eta(t) \Big[   \int_{\p \Om} (\Phi \cdot \nu) \mathcal{T} \Big( (\rho^{\T_k})^{\al- \frac{1}{2}}(t) \Big) d\mathcal{H}^{d-1}  \Big]dt
\ra 
\int_0^T \eta(t)\Big[
\int_{\p \Om} (\Phi \cdot \nu) \mathcal{T} \Big(\rho^{\al- \frac{1}{2}}(t) \Big) d\mathcal{H}^{d-1} \Big]dt
$$
as $k \ra \infty$.
Since it holds from \eqref{sobolev}  that  $\mathcal{T}\big( (\rho^{\T_k})^{\al- \frac{1}{2}}(t) \big)=\lam^{\al - \frac{1}{2} }$, we deduce
$$
 \int_{\p \Om} (\Phi \cdot \nu)
 \lam^{\al- \frac{1}{2}} d\mathcal{H}^{d-1}  
=
\int_{\p \Om} (\Phi \cdot \nu) \mathcal{T} \Big(\rho^{\al- \frac{1}{2}}(t) \Big) d\mathcal{H}^{d-1}
$$
for a.e. $t \in (0,T)$. Hence, $\mathcal{T}\big(\rho^{ \al - \frac{1}{2}}(t)\big)=\lam^{\al - \frac{1}{2}}$, since $\Phi \in C_b^\infty(\OM, \R^d)$ is arbitrary. Therefore, we conclude $\rho^{\al - \frac{1}{2}}(t) -\lam^{\al - \frac{1}{2}} \in H_0^1(\Om)$ and we get the desired result.   $\Box$

\subsection{Curve of maximal slope}

We prove part (iv) by using the De Giorgi variational interpolation.
For $\eta\in\mathcal M_2(\Omega)$ with $\calE(\eta)<\infty$ and $\delta>0$, we set
\begin{equation}\label{DG_interpolation}
J_\delta[\eta]
:=
\argmin_{\nu\in\mathcal M_2(\Omega)}
\left\{
\calE(\nu)+\frac{Wb_2^2(\eta,\nu)}{2\delta}
\right\}.
\end{equation}
By Lemma~\ref{minimizer}, the set $J_\del[\eta]$ consists of a unique element. For
$t\in(k\T,(k+1)\T]$, we write $\delta=t-k\T$ and define
\[
\widetilde\rho^\T(t)\in J_\delta[\rho_k^\T],
\quad \textrm{with} \quad
\widetilde\rho^\T(0):=\rho_0.
\]
Since the minimizer is unique, combining 
\cite[Lemma~3.2.2 and Eq.~(3.2.3)]{AGS}, we obtain
\begin{equation}\label{almost_CMS}
\calE(\rho_m^\T)-\calE(\rho_n^\T)
\geq
\frac12\int_{m\T}^{n\T}|\Th^\T(t)|^2\,dt
+
\frac12\int_{m\T}^{n\T}
|\na\calE|^2(\widetilde\rho^\T(t))\,dt
\end{equation}
for all $m<n$, $~m,n\in\N\cup\{0\}$, where
\[
\Th^\T(t)
=
\frac{Wb_2(\rho_k^\T,\rho_{k+1}^\T)}{\T}
\qquad
\text{if }~ t\in(k\T,(k+1)\T].
\]

Fix $T>0$ and let $(\T_k)$ be the sequence obtained in (i). 
From \eqref{almost_CMS}, we get
\begin{equation}\label{uniformly_L^2}
\int_{0}^{T} 
|\Th^{\T_k} (t)|^2 dt 
+ \int_{0}^{T}  |\na \calE|^2 (\widetilde \rho^{\T_k} (t))dt \leq 2\calE(\rho_0).
\end{equation}
Using \eqref{uniformly_L^2}, there exists a subsequence, not relabeled, such that $\Th^{\T_k} \wra \Th$ in $L^2([0,T])$. From the definition of $\Th^{\T_k}$, we obtain
$$
Wb_2(\rho^{\T_k}(t) , \rho^{\T_k}(s)) 
\leq \int_{ \lceil \frac{s}{\T_k} \rceil \T_k }^{ \lceil \frac{t}{\T_k}  +1  \rceil \T_k}  \Th^{\T_k}(r)dr \quad \forall t\geq s,
$$
which implies by \eqref{convergence1}
$$
Wb_2(\rho(t), \rho(s)) 
\leq \limsup_{k\ra \infty} Wb_2(\rho^{\T_k}(t) , \rho^{\T_k}(s))
\leq \int_s^t \Th(r)dr.
$$
Hence, it follows that 
$$
\mu \in  AC_{loc}^2([0,\infty), \mathcal{M}_2(\Om)),
$$ 
since $T>0$ is arbitrary.  Moreover, it is direct that $|\mu'|(t) \leq \Th(t)$ for a.e. $t\geq0$ and $|\mu'| \in L^2([0,T])$. 
Thanks to the weak convergence of $\Th^{\T_k}$ to $\Th$ and the elementary inequality $a^2-b^2 \geq 2b(a-b)$, we get
\begin{equation}\label{liminf1}
\liminf_{k \ra \infty}  \int_{t_1}^{t_2} 
|\Th^{\T_k} (t)|^2 dt 
\geq  \int_{t_1}^{t_2} 
|\Th(t)|^2 dt 
\geq  \int_{t_1}^{t_2} 
|\mu' (t)|^2 dt.
\end{equation}

Let $0\leq t_1 < t_2 <T$. Clearly, there exist uniquely determined $m_k, n_k \in \N \cup \{ 0\}$ such that $n_k \T_k < t_2 \leq (n_k +1) \T_k$ and  $(m_k-1) \T_k < t_1 \leq m_k \T_k$. 
Applying \eqref{almost_CMS} with $m=m_k$ and $n=n_k+1$, we get
$$
\calE(\rho_{m_k}^{\T_k}) 
\geq  \calE(\rho_{n_k+1}^{\T_k})  
+ \frac{1}{2} \int_{m_k \T_k}^{(n_k+1) \T_k} 
|\Th^{\T_k} (t)|^2 dt 
+ \frac{1}{2} \int_{m_k \T_k}^{(n_k+1) \T_k}  |\na \calE|^2 (\widetilde \rho^{\T_k} (t))dt.
$$
Set
\(
I_k:=\big(m_k\T_k,(n_k+1)\T_k\big]
\quad\text{and}\quad
I:=(t_1,t_2].
\)
Since $m_k\T_k\to t_1$ and $(n_k+1)\T_k\to t_2$, we have
$\mathbf{1}_{I_k}\to\mathbf{1}_{I}$ almost everywhere. Moreover, the weak convergence of $\Th^{\T_k}$ to $\Th$ in $L^2([0,T])$ implies
\[
\mathbf{1}_{I_k}\Th^{\T_k}
\rightharpoonup
\mathbf{1}_{I}\Th
\quad\text{in }L^2([0,T]).
\]
Therefore, by weak lower semicontinuity and Fatou's lemma, we obtain
\begin{align}
&\liminf_{k\ra\infty}\calE(\rho_{m_k}^{\T_k})
\nonumber \\
&\geq
\liminf_{k\ra\infty}\calE(\rho_{n_k+1}^{\T_k})
+\frac12\int_{t_1}^{t_2}|\Th(t)|^2 dt
+\frac12\int_{t_1}^{t_2}
\liminf_{k\ra\infty}
|\na\calE|^2(\widetilde \rho^{\T_k} (t)) dt.
\label{triple_liminf}
\end{align}
Here, we applied Fatou's lemma to
$\mathbf 1_{I_k}|\nabla\mathcal E|^2(\widetilde \rho^{\tau_k} )$ to get
\[
\liminf_{k\to\infty}\int_{I_k}
|\nabla\mathcal E|^2(  \widetilde \rho^{\tau_k} (t))dt
\ge
\int_I\liminf_{k\to\infty}
|\nabla\mathcal E|^2( \widetilde \rho^{\tau_k}(t))dt.
\]

We next compare the two interpolations. If
$t\in(k\T,(k+1)\T]$ and $\delta=t-k\T$, the minimality in
\eqref{DG_interpolation}, together with $\calE\geq0$, gives
\[
\calE(\widetilde\rho^\T(t))
+
\frac{Wb_2^2(\rho_k^\T,\widetilde\rho^\T(t))}{2\delta}
\leq
\calE(\rho_k^\T)
\leq
\calE(\rho_0).
\]
Moreover, the one-step estimate in \eqref{bound1} yields
\[
Wb_2^2(\rho_k^\T,\rho_{k+1}^\T)
\leq 2\T\calE(\rho_0).
\]
Since $\rho^\T(t)=\rho_{k+1}^\T$, it follows that
\begin{equation}\label{DG_close}
\sup_{t\geq0}
Wb_2(\widetilde\rho^\T(t),\rho^\T(t))
\leq 2\sqrt{2\T\calE(\rho_0)},
\qquad
\calE(\widetilde\rho^\T(t))\leq\calE(\rho_0).
\end{equation}
Combining \eqref{DG_close} with \eqref{convergence1}, we obtain
\begin{equation}\label{DG_convergence}
Wb_2(\widetilde\rho^{\T_k}(t),\rho(t))\to0
\qquad\ \forall t\geq0,
\end{equation}
which gives
\begin{equation}\label{liminf2}
\liminf_{k \ra \infty} |\na \calE|(\widetilde \rho^{\T_k} (t))
\geq
|\na^- \calE|(\rho(t))
\qquad\ \forall t\geq0.
\end{equation}


Next, we observe that $Wb_2(\rho_{n_k +1}^{\T_k} ,\rho(t_2)) \leq Wb_2(\rho_{n_k +1}^{\T_k} , \rho^{\T_k}(t_2)) + Wb_2(\rho^{\T_k}(t_2), \rho(t_2))  \ra 0$. 
Since $\calE$ is lower semicontinuous with respect to $Wb_2$
as mentioned in the proof of Lemma \ref{minimizer}, we have
\begin{equation}\label{liminf3}
\liminf_{k \ra \infty} \calE(\rho_{n_k+1}^{\T_k}) \geq \calE(\rho(t_2)).
\end{equation}

\noindent
Thanks to \eqref{convergence2}, we have 
\begin{equation}\label{strong_liminf}
\liminf_{k \ra \infty} \calE(\rho_{m_k}^{\T_k}) = \calE(\rho(t_1)), \quad \mathcal{L}^1\textrm{-}a.e. ~ t_1 \geq0.
\end{equation}

Combining \eqref{liminf1}-\eqref{strong_liminf}, since $T$ is arbitrary,
we conclude that there exists a $\mathcal{L}^1$-negligible set $\calN \subset [0,\infty)$ such that for all $0\leq t_1 \leq t_2$ with $t_1 \notin \calN$,  
\begin{equation*}
{\calE}(\rho(t_1)) \geq {\calE}(\rho(t_2)) + \frac{1}{2}\int_{t_1}^{t_2}  |\mu'|^2(t) dt
+ \frac{1}{2} \int_{t_1}^{t_2} |\na^- \calE(\rho(t))|^2 dt.
\end{equation*} 
Since $\rho_0=\rho_0^\T$, we know $0 \notin \calN$. Then, the function
$$
\vphi(t):=\inf_{s \in [0,t]\backslash \calN} \calE(\rho(s)), \quad t\geq0,
$$
is well-defined and non-increasing.
Therefore, $\vphi(t)=\calE(\rho(t))$ up to $\calN$, and
\eqref{eq:CMS} is satisfied for $\vphi$ by the Lebesgue differentiation theorem.
This completes the proof. \qed

\vspace{3mm}
\appendix
\section{Extension to the drift-diffusion case}
It is worth noting that Theorem \ref{main} can be extended to the drift-diffusion case, in a direction parallel to \cite[Section 4]{FG}. For $\al>1$ and $\lam \ge0$, by considering the free energy functional
\begin{equation*}
\displaystyle{\int_{\Om}}  
  \lt(
  \frac{\rho^{\al} - \al \lam^{\al-1}\rho}{\al-1}   + \lam^\al 
   - V\rho \rt)\,dx,
\end{equation*}
with minimizing movement schemes, we can obtain a weak solution of
\begin{equation*}
  \begin{cases}
                                   \p_t \rho= \Delta \rho ^{\al} - \textrm{div} (\rho \na V) \\ 
                                   \rho(0)=\rho_0  
  \end{cases}
\end{equation*}
subject to the Dirichlet boundary condition
\begin{equation}\label{eq:dd:D}
\rho|_{\partial \Om} = \lt(\lam^{\al-1} + \frac{(\al -1)V}{\al } \rt)^{\frac{1}{\al-1}},
\end{equation}
where $V: \overline{\Om} \to \R$ is a given smooth potential function such that the right-hand side of \eqref{eq:dd:D} is non-negative.

We note that the Dirichlet boundary condition arises as the minimizer of the entropy integrand $\rho \mapsto  \frac{\rho^{\al} - \al \lam^{\al-1}\rho}{\al-1}   + \lam^\al 
   - V\rho$. An additional condition $(\lam^{\al-1} + \frac{(\al -1)V}{\al } )|_{\partial \Om}\ge 0$ is imposed to ensure the non-negativity of $\rho$ on $\p \Om$.
In particular, if $V$ is compactly supported, then the Dirichlet boundary condition \eqref{eq:dd:D} reduces to the constant value
$$
\rho|_{\p \Om} = \lambda \geq 0.
$$

On the other hand, by formally taking $\al \to 1^+$, we recover the Dirichlet boundary condition corresponding to the heat equation case (see \cite[Section 4]{FG}):
$$
    \rho|_{\partial \Om} = \lambda e^V.
$$


\section{Lower bounds on metric derivatives}
We investigate pointwise lower bounds and vanishing of $|\mu'|$.

\begin{proposition}\label{sub}
Let $\mu(t)=\rho(t) \mathcal{L}^d |_{\Om}$ be the curve in Theorem \ref{main}. Then we have
\begin{enumerate}
\item[\emph{(i)}]
Let $v(t):= -\na \rho^\al(t) / \rho(t)$. The metric derivative satisfies
\begin{equation}\label{metric_derivative_BB}
 |\mu '|(t) \geq ||v(t)||_{L^2(\mu_t; \Om)}
\quad \mathcal{L}^1\textrm{-a.e.}~ t > 0
\end{equation}
whenever
\begin{equation*}
v(t) \in \overline{\{ \na \phi : \phi \in C_c^\infty(\Om) \}}^{L^2(\mu_t; \Om)}.
\end{equation*}

\item[\emph{(ii)}]
For $\mathcal{L}^1$-a.e. $t>0$,
$||v(t)||_{L^2(\mu_t; \Om)}= 0$ implies that
 $|\mu'|(t)=0$.

\end{enumerate}
\end{proposition}

\begin{proof}
Let $\z \in C_c^{\infty}(\Om)$. Given a measure $\mu \in \mathcal{M}_2(\Om)$, we denote $\mu(\z):=\int_\Om \z(x) d\mu(x)$.
For each $t_0 \geq 0$ we have
$$
\rho_{  t_0 + h}(\z) - \rho_{t_0}(\z) = \int_{\Om} \z(x) d\rho_{t_0+h}(x) - \int_{\Om} \z(y) d\rho_{t_0}(y)
=\int_{\OM \times \OM} \z(x) - \z(y) d\gam^h(x,y),
$$
where $ \gam^h \in OPT(\rho_{t_0+h}, \rho_{t_0})$. We define
\begin{equation}\nonumber
H(x,y):=
  \begin{cases}
                                    |\na \z(x)| & \text{if $x=y$,} \\
                                    \\
                                   \displaystyle{\frac{|\z(x) - \z(y)|}{|x-y|} }  & \text{if $x \neq y$.}
  \end{cases}
\end{equation}
Observe that
$$
\Big|  \rho_{t_0 +h}(\z) - \rho_{t_0}(\z) \Big| \leq \int_{\OM \times \OM} |x-y| H(x,y)d\gam^h
$$
\begin{equation}\label{rho_difference}
\leq Wb_2(\rho_{t_0 +h}, \rho_{t_0}) \Big(  \int_{\OM \times \OM} H^2(x,y) d\gam^h \Big)^{1/2}.
\end{equation}
Let $(h_j)_{j \in \N}$ be a sequence such that $h_j \da 0$ and $\gam^{h_j}(\OM \times \OM) \leq \rho_{t_0 +h_j}(\Om) + \rho_{t_0}(\Om)\leq 2 m_\infty$. Then it follows that $(\gam^{h_j})$ is weakly relatively compact in duality with  $C_c(\OM \times \OM \backslash \p \Om \times \p \Om)$ (see Remark \ref{no_need_bdry^2}).
Passing to a subsequence, we may assume that $\gam^{h_j} \wra \gam_{t_0}$ in duality with  $C_c(\OM \times \OM \backslash \p \Om \times \p \Om)$. Since $H^2$ is bounded, upper semicontinuous, and compactly supported in $\OM \times \OM \backslash \p \Om \times \p \Om$, the Portmanteau theorem yields
\begin{equation}\label{H_convergence}
\limsup_{j \ra \infty} \int_{\OM \times \OM} H^2(x,y) d\gam^{h_j} 
\leq \int_{\OM \times \OM} H^2(x,y) d\gam_{t_0}
\end{equation}

It follows from the weak lower semicontinuity of $\gam \mapsto \mathcal{C}(\gam)$  and \eqref{HolderConti2} that
$$
\int_{\OM \times \OM} |x-y|^2 d\gam_{t_0} =0 \quad \textrm{with} \quad (\pi^i_\# \gam_{t_0}) |_{\Om} = \rho_{t_0},
\quad i=1,2,
$$
which means that
\begin{equation}\label{diagonal_support}
\textrm{supp} (\gam_{t_0}) \subset \{  (x,x): x \in \OM \}.
\end{equation}
Thanks to \eqref{rho_difference}, \eqref{H_convergence}, and \eqref{diagonal_support}, we have
$$
\limsup_{j \ra \infty} \Big| \frac{\rho_{t_0 +h_j}(\z) - \rho_{t_0}(\z) }{h_j} \Big|
 \leq |\mu'|(t_0) \Big(\int_{\OM \times \OM} H^2(x,y) d\gam_{t_0} \Big)^{1/2}
$$
\begin{equation}\label{riesz1}
=  |\mu'|(t_0)  \Big(\int_{\Om \times \OM} |\na \z(x)|^2 d\gam_{t_0} \Big)^{1/2}  
=  |\mu'|(t_0)  \Big(\int_{\Om} |\na \z(x)|^2 d\rho_{t_0}(x) \Big)^{1/2}. 
\end{equation}

On the other hand, using \eqref{weak_solution} and \eqref{riesz1}, we have for a.e. $t_0 > 0$
\begin{equation}\label{LDT+Optimal_plan}
\Big| \int_\Om \na \rho^\al(t_0) \cdot \na \z dx  \Big|
=\Big| \int_\Om \rho^\al(t_0)  \Delta \z  dx \Big|
\leq |\mu'|(t_0)  || \na \z ||_{L^2(\rho_{t_0} dx)}
\end{equation}
where we used the fact that $\rho^\al(t_0) \in W^{1,1}(\Om)$.

Let $v(t_0) \in \overline{\{ \na \phi : \phi \in C_c^\infty(\Om) \}}^{L^2(\mu_{t_0})}$.
Then we get a sequence $(\z_n) \subset C_c^\infty (\Om)$, depending on $t_0$, such that 
\begin{equation}\nonumber
|| \na \rho^{\al-1}(t_0) - \na \z_n  ||_{L^2(\mu_{t_0})} \ra 0    \quad \textrm{as $n \ra \infty $},
\end{equation}
which implies 
\begin{equation}\label{BB_3}
\int_\Om |\na \z_n|^2 \rho(t_0)  dx
\ra 
\int_\Om |\na \rho^{\al-1}(t_0)|^2 \rho(t_0) dx,
\end{equation}

\begin{align}
&\int_\Om \na \rho^\al(t_0) \cdot \na \z_n dx 
= \int_\Om \frac{\al}{\al-1} \na \rho^{\al-1}(t_0) \sqrt{\rho(t_0)}
\cdot \na \z_n \sqrt{\rho(t_0)} dx
  \nonumber \\
&  \ra 
\int_\Om \frac{\al}{\al-1} \na \rho^{\al-1}  (t_0)\sqrt{\rho(t_0)}
\cdot
\na \rho^{\al-1}(t_0) \sqrt{\rho(t_0)} dx 
=\int_\Om \na \rho^\al(t_0) \cdot \na \rho^{\al-1}(t_0) dx 
\label{BB_4}
\end{align}
as $n \ra \infty$.

Combining \eqref{LDT+Optimal_plan} with \eqref{BB_3}-\eqref{BB_4}, we finally obtain 
$$
|\mu'|(t_0) \geq \frac{ \int_\Om \na \rho^\al(t_0) \cdot \na \rho^{\al-1}(t_0) dx}{\sqrt{\int_\Om |\na \rho^{\al-1}(t_0)|^2 \rho(t_0) dx}} = \Big(\int_\Om \frac{|\na \rho^\al (t_0)|^2 }{\rho(t_0)} dx \Big)^{1/2} 
=  \Big(\int_\Om   |v(t_0)|^2 d\mu_{t_0}(x) \Big)^{1/2}
$$
whenever $ || \na \rho^{\al-1}(t_0) ||_{L^2(\rho_{t_0} dx)} \neq 0$, or equivalently, $ || v(t_0) ||_{L^2(\mu_{t_0}; \Om)} \neq 0$. Hence, \eqref{metric_derivative_BB} holds.

\vspace{3mm}
It remains to show (ii).
If we assume $ || v(t_0) ||_{L^2(\mu_{t_0}; \Om)} = 0$, then $\rho(t_0)$ is constant a.e. As $\rho^{\al-\frac{1}{2}}(t_0) - \lam^{\al-\frac{1}{2}} \in H_0^1 (\Om)$, we conclude that $\rho(t_0)=\lam$ a.e.
Therefore, the energy functional $\calE$ attains its global minimum, that is, ${\calE}(\mu(t_0))=0$.
On the other hand, since
$\mu \in  AC_{loc}^2([0,\infty), \mathcal{M}_2(\Om))$, we know that $|\mu'|(t_0)$ exists.
Note that ${\calE}(\mu(t)) \leq {\calE}(\mu(s))$ for $\mathcal{L}^1$-a.e. $s,t$ with $s \leq t$.
Therefore, we see that $Wb_2(\mu(t_0), \mu(t_0 +h))=0$ for a.e. $h>0$, which yields that $|\mu'|(t_0)=0$. This completes the proof.  
\end{proof}


\vspace{3mm}
\section*{Acknowledgments}
D. Kim is supported by the National Research Foundation of Korea(NRF) grant funded by the Ministry of Education (RS-2024-00450030) and the Korea government(MSIT) (grant No. 2022R1A4A1032094).
D. Koo is supported by NRF grant 2022R1A2C1002820 and RS-2024-00406821.
G. Seo was supported by the National Research Foundation of Korea (NRF) grant funded by the Korea government(MSIT) (RS-2023-00219980 and RS-2023-00212227).


\end{document}